\documentclass[11pt]{article}

\usepackage{amsmath,amsthm}
\usepackage{amssymb,latexsym}

\usepackage{enumerate}

\newcommand{\AAA}{{\cal A}}
\newcommand{\BB}{{\cal B}}
\newcommand{\DD}{{\cal D}}
\newcommand{\EE}{{\cal E}}
\newcommand{\FF}{{\cal F}}

\newcommand{\HH}{{\cal H}}
\newcommand{\LL}{{\cal L}}
\newcommand{\MM}{{\cal M}}

\newcommand{\VV}{{\cal V}}
\newcommand{\WW}{{\cal W}}

\newcommand{\BN}{{\mathbb N}}
\newcommand{\BR}{{\mathbb R}}
\newcommand{\BX}{{\mathbb X}}
\newcommand{\BBM}{{\mathbf M}}

\newcommand{\BBX}{{\mathbf X}}

\newcommand{\fch}{{\mathbf{1}}}

\newcommand{\dist}{\mbox{\rm dist}}

\newtheorem{theorem}{\bf Theorem}[section]
\newtheorem{proposition}[theorem]{\bf Proposition}
\theoremstyle{definition}
\newtheorem{definition}[theorem]{Definition}
\newtheorem{example}[theorem]{\bf Example}

\newtheorem{remark}[theorem]{Remark}

\numberwithin{equation}{section}

\begin{document}

\title {Systems of semilinear parabolic variational inequalities with
time-dependent convex obstacles}
\author {Tomasz Klimsiak, Andrzej Rozkosz and Leszek S\l omi\'nski}
\date{}
\maketitle
\begin{abstract}
We consider a system of seminlinear parabolic variational
inequalities with time-dependent convex obstacles. We prove the
existence and   uniqueness  of its solution. We also provide a
stochastic representation of the solution and show that it can be
approximated by the penalization method. Our proofs are based upon
probabilistic methods from the theory of Markov processes and the
theory of backward stochastic differential equations.
\end{abstract}


\footnotetext{T. Klimsiak: Institute of Mathematics, Polish
Academy of Sciences, \'Sniadeckich 8, 00-956 Warszawa, Poland, and
Faculty of Mathematics and Computer Science, Nicolaus Copernicus
University, Chopina 12/18, 87-100 Toru\'n, Poland. e-mail:
tomas@mat.umk.pl}

\footnotetext{A. Rozkosz and L. S\l omi\'nski: Faculty of
Mathematics and Computer Science, Nicolaus Copernicus University,
Chopina 12/18, 87-100 Toru\'n, Poland. e-mails: rozkosz@mat.umk.pl
(A. Rozkosz), leszeks@mat.umk.pl (L. S\l omi\'nski).}

\section{Introduction}
\label{sec1}

Let $E$ a bounded domain in $\BR^d$, $E_T=[0,T]\times E$ and let
$D=\{D(t,x),(t,x)\in E_T\}$ be a family of uniformly bounded
closed convex sets in $\BR^m$ with nonempty interiors. In the
paper we study the problem, which roughly speaking can be stated
as follows: for measurable function $\varphi:E\rightarrow\BR^m$
such that $\varphi(x)\in D(T,x)$, $x\in E$, and measurable
$f:E_T\times\BR^m\times\BR^{m\times d}\rightarrow \BR^m$ find
$u=(u^1,\dots,u^m):E_T\rightarrow\BR^m$ such that
\begin{equation}
\label{eq1.1} u(t,x)\in D(t,x)\mbox{\, for\, } (t,x)\in E_T, \quad
u(T, \cdot)=\varphi, \quad u(t, \cdot)_{|\partial E}=0,\,\,t
\in(0,T)
\end{equation}
and
\begin{equation}
\label{eq1.2} \int^T_0\Big(\frac{\partial u}{\partial t}
+L_tu+f_u,v-u\Big)_H\,dt\ge0
\end{equation}
for every $v=(v^1,\dots,v^m)$ such that $v^i\in
L^2(0,T;H^1_0(E))$, $i=1,\dots,m$, and  $v(t,x)\in D(t,x)$ for
$(t,x)\in E_T $. In (\ref{eq1.2}), $(\cdot,\cdot)_H$ is the usual
inner product in $H=[L^2(E)]^m$,
\begin{equation}
\label{eq1.4} f_u(t,x)=f(t,x,u(t,x),\sigma\nabla
u^1(t,x),\dots,\sigma\nabla u^m(t,x)), \quad (t,x)\in E_T
\end{equation}
and
\[
L_{t}=\frac 12 \sum_{i,j=1}^{d} \frac{\partial}{\partial
x_{j}}\Big(a_{ij}(t,x)\frac{\partial}{\partial x_{i}}\Big),
\]
where $a: E_T\rightarrow\BR^d\otimes\BR^d$ is a measurable
symmetric matrix-valued function  such that for some
$\Lambda\ge1$,
\begin{equation}
\label{eq1.3} \Lambda^{-1}|y|^2\le\sum_{i,j=1}^{d} a_{ij}(t,x)
y_{i}y_{j}\le\Lambda |y|^{2},\quad y\in \mathbb{R}^{d}
\end{equation}
for a.e. $(t,x)\in E_{T}$.  By putting $a(t,x)=a(0,x)$ for
$t\notin[0,T]$, $x\in E$ we can and will assume that $a$ is
defined and satisfies (\ref{eq1.3}) in all $\BR\times E$. In
(\ref{eq1.4}), $\sigma$ is  the symmetric square root of $a$.

The main feature of the paper is that we deal with time-dependent
obstacles. In case of single equation, i.e. when $m=1$, problem
(\ref{eq1.1}), (\ref{eq1.2}) is quite well investigated. For
various results on existence, uniqueness and approximation of
solutions in case of $L^2$ data and one or two regular obstacles,
i.e. when $D$ has the form $D(t,x)=\{y\in\BR:\underline h(t,x)\le
y\le\bar h(t,x)\}$ for some regular $\underline h ,\bar
h:E_T\rightarrow\bar\BR$ (possibly $\underline h\equiv-\infty$ or
$\bar h\equiv+\infty$) see the monograph \cite[Sections 2.2,
2.18]{BL} and more recent papers \cite{C,K,KY}. Linear problem of
the form (\ref{eq1.1}), (\ref{eq1.2}) with $L^2$ data and one
irregular barrier is investigated in \cite{MP,P}. For recent
results on semilinear problem see \cite{K:SPA} (one merely
measurable obstacle) and \cite{K:PA} (two measurable obstacles
satisfying some separation condition). The problem with two
irregular obstacles and $L^1$ data is investigated in
\cite{KR:JEE}.

In case of systems of equations the situation is quite different.
To our knowledge,  in this case only few partial results exist
(see \cite[Section 1.2]{MP} for the existence of solutions of
weakly coupled systems and Example 9.3 and Theorem 9.2 in
\cite[Chapter 2]{L2} for the special case $0\in
D(t_2,\cdot)\subset D(t_1,\cdot)$ if $0\le t_1\le t_2$; see also
\cite{KSY} for existence results concerning  a different but
related problem). The aim of the present paper is to prove quite
general results on existence, uniqueness and approximation of
solutions of (\ref{eq1.1}), (\ref{eq1.2}) in case the data are
square integrable and $D$ satisfies some mild regularity
assumptions. The case of $L^1$ data and irregular obstacles is
more difficult but certainly deserves further investigation.
%We think that one reason for lack of satisfactory theory for
%systems of equations is the problem with proper definition of
%solutions. More precisely, the problem is to provide a definition
%which under some natural assumptions on $\varphi,f$ and $D$ ...

In our opinion one of the main problem one encounters when dealing
with systems and time-dependent obstacles lies in the proper
choice of the definition of a solution. In fact, the main problem
is to adopt definition which ensures uniqueness of solutions. The
definition used in the present paper is a natural extension to
systems of the definition used in one-dimensional case in \cite{P}
and then in \cite{K:SPA,K:PA,KR:JEE}. By a solution of
(\ref{eq1.1}), (\ref{eq1.2}) we mean a pair $(u,\mu)$ consisting
of a  function $u=(u^1,\dots,u^m):E_T\rightarrow\BR^m$ and vector
$\mu=(\mu^1,\dots\mu^m)$ of signed Borel measures on $E_T$
satisfying the following three conditions:
\begin{enumerate}
\item[(a)]
$u^i$ are quasi-continuous (with respect to the parabolic capacity
determined by $L_t$) functions of class $C([0,T];H)\cap
L^2(0,T;H^1_0(E))$ and $\mu^i$ are smooth (with respect to the
same capacity) measures of finite variation,

\item[(b)] $u$ is a weak solution of the problem
\[
\frac{\partial u}{\partial t}+L_tu=-f_u-\mu,\quad u(T,
\cdot)=\varphi, \quad u(t, \cdot)_{|\partial E}=0,\,\,t \in(0,T),
\]
such that $u(t,x)\in D(t,x)$ for quasi-every (q.e. for short)
$(t,x)\in(0,T]\times E$,
\item[(c)]
for every quasi-continuous function $h=(h^1,\dots,h^m)$ such that
$h(t,x)\in D(t,x)$ for q.e $(t,x)\in(0,T]\times E$,
\[
\sum^m_{i=1}\int^T_t\!\!\int_E(u^i-h^i)\,d\mu^i\le0,\quad
t\in(0,T).
\]
\end{enumerate}
In the above definition  $\mu$ may be called the ``obstacle
reaction measure". It may be interpreted as the energy we have to
add to the system to keep the solution inside $D$. Condition (c)
is some kind of minimality condition imposed on $\mu$. In case
$m=1$ it reduces to the usual minimality condition saying that
$\mu=\mu^{+}-\mu^{-}$, where $\mu^+$ (resp. $\mu^{-}$) is a
positive measure acting only when $u$ is equal to the lower
obstacle $\underline h$ (resp. upper obstacle $\bar h$). Also
remark that an important requirement in our definition is that $u$
is quasi-continuous and $\mu$ is smooth. It not only ensures that
the integral in (c) is meaningful, but also allows us to give a
probabilistic representation of solutions. In fact, this
probabilistic representation  may serve as an equivalent
definition of a solution of (\ref{eq1.1}), (\ref{eq1.2}).
%It not only ensures that the integral in (c) is meaningful but
%also allows us to give an equivalent probabilistic definition of a
%solution.

As in classical monographs \cite{BL,DL,L1,L2}, and papers
\cite{KSY,KY,MP}, in the present paper we work in $L^2$ setting.
We assume that $\varphi\in [L^2(E)]^m$, $f(\cdot,\cdot,0,0)\in
L^2(0,T;[L^2(E)]^m)$ and $f(t,x,\cdot,\cdot)$ is Lipschitz
continuous for $(t,x)\in E_T$. An important, model example of
operator $L_t$ satisfying (\ref{eq1.3}) is the Laplace operator.
But as in \cite{L1,L2,MP}, to cover classic examples, like
temperature control in domains with discontinuous coefficient of
thermal conductivity (see \cite[Chapter 1, \S3.4]{L1},
\cite[Chapter I, \S3.3, 4.4]{DL}), in the paper we consider
divergence form operator with possibly discontinuous $a$. As for
$D$, we assume that $(t,x)\mapsto D(t,x)\in\mbox{Conv}$ is
continuous if we equip $\mbox{Conv}$ with the Hausdorff metric. We
also assume that $D$ satisfies the following separation condition:
one can find a solution $u_{*}\in\WW$ of the Cauchy problem
\begin{equation}
\label{eq1.5} \frac{\partial u_{*}}{\partial
t}+L_tu_{*}=-f_{*},\quad u_{*}(T)=\varphi_{*}
\end{equation}
with some $\varphi_{*}\in[L^2(E)]^m$, $f_{*}\in
L^2(0,T;[L^2(E)]^m)$ such that $u_{*}(t,x)\in D^{*}(t,x)$ for q.e.
$(t,x)\in E_T$, where $D^{*}(t,x)=\{y\in D(t,x):\dist(y,\partial
D(t,x))\ge\varepsilon\}$ for some $\varepsilon>0$. We show that
under the above assumptions there exists a unique solution
$(u,\mu)$ of (\ref{eq1.1}), (\ref{eq1.2}) and that $u$ and $\mu$
may be approximated by the penalization method.  Note that our
separation condition is not optimal, because we assume that
$\varepsilon>0$ and that $u_{*}$ is more regular than  the
solution $u$ itself (see condition (a)). The condition is also
stronger than known sufficient separation conditions in the one
dimensional case (see \cite{K:PA}). Nevertheless, it is satisfied
in  many interesting situations.

As in \cite{K:SPA,K:PA,KR:JEE}, to prove our result we use
probabilistic methods. In particular, we rely heavily on the
results of our earlier paper \cite{KRS} devoted to reflected
backward stochastic differential equations with time-dependent
obstacles and in proofs we use the  methods of the theory of
Markov processes and probabilistic potential theory. Also note
that the first results on multidimensional reflected backward
stochastic differential equations were proved in  \cite{GP} in
case $D$ is a fixed convex domain. The results of \cite{GP} were
generalized in \cite{Ou} to equations with Wiener-Poisson
filtration. For related results with some  time-depending domains
see the recent paper \cite{NO}.

\section{Preliminaries}
\label{sec2}

For $x\in\BR^m$, $z\in\BR^{m\times d}$ we set
$|x|^2=\sum^m_{i=1}|x_i|^2$, $\|z\|^2=\mbox{trace}(z^*z)$. By
$\langle\cdot,\cdot\rangle$ we denote the usual scalar product in
$\BR^m$. Given a Hilbert space $\HH$ we denote by $[\HH]^m$ its
product equipped with the usual inner product
$(u,v)_{[\HH]^m}=\sum^m_{i=1}(u^i,v^i)_{\HH}$ and norm
$\|u\|_{[\HH]^m}=((u,u)_{[\HH]^m})^{1/2}$. We will identify the
space $[L^2(0,T;\HH)]^m$ with $L^2(0,T;[\HH]^m)$.

Throughout the paper, $E$ is a nonempty bounded connected open
subset of $\BR^d$. $\bar E$ is the closure of $E$ in $\BR^d$,
$E^1=\BR\times E$, $E_T=[0,T]\times E$, $E_{0,T}=(0,T]\times E$.
We set $H=[L^2(E)]^m$, $V=[H^1_0(E)]^m$, $V'=[H^{-1}(E)]^m$, where
$H^{-1}(E)$ is the dual of the Sobolev space $H^1_0(E)$. For $u\in
V$ we set $|\!|\!|\nabla
u|\!|\!|^2_H=\sum^d_{k=1}\sum^m_{i=1}\|\frac{\partial
u^i}{\partial x_k}\|^2_{L^2(E)}$.

The Lebesgue measure on $\BR^d$ will be denoted by $m$. By $m_1$
we denote the Lebesgue measure on $E^1$.

\subsection{Convex sets and functions}

By $\mbox{\rm Conv}$ we denote the space of all bounded closed
convex subsets of $\BR^m$ with nonempty interiors endowed with the
Hausdorff metric $\rho$, that is for any $D,G\in\mbox{\rm Conv}$
we set
\[
\rho(D,G)=\max\big(\sup_{x\in D}\dist(x,G),\sup_{x\in G}
\dist(x,D)\big),
\]
where $\dist(x,D)=\inf_{y\in D}|x-y|$.

Let $D\in\mbox{\rm Conv}$ and let   ${\cal N}_{y}$ denote the set
of inward normal unit vectors at $y\in\partial D$. It is well
known (see, e.g., \cite{M}) that ${\bf n}\in{\cal N}_{y}$ if and
only if $\langle y-x,\mbox{\bf n}\rangle\leq 0$ for every $x\in
D$. If moreover $a\in\mbox{\rm Int} D$ then for every  ${\bf n}
\in{\cal N}_{y}$,
\begin{equation}
\label{eq2.3} \langle y-a,\mbox{\bf n}\rangle\leq-\dist(a,\partial
D).
\end{equation}
If $\dist(x,D)>0$ then there exists a unique
$y=\Pi_D(x)\in\partial D$ such that $|y-x|=\dist(x,D)$. One can
observe that $(y-x)/|y-x|\in{\cal N}_{y}$. Moreover (see
\cite{M}), for every $a\in\mbox{\rm Int} D$,
\begin{equation}
\label{eq2.6} \langle x-a,y-x\rangle\leq-\dist(a,\partial D)|y-x|.
\end{equation}
Also note that for any nonempty bounded closed convex sets
$D,G\subset\BR^m$ and any $x,y\in\BR^m$,
\begin{equation}
\label{eq2.1}
|\Pi_D(x)-\Pi_G(y)|^2\le|x-y|^2+2[\dist(x,D)+\dist(y,G)]\rho(D,G)
\end{equation}
(see \cite[Chapter 0, Proposition 4.7]{MM}).

\subsection{Time-dependent Dirichlet forms}

Let $\AAA=\{a(t;\cdot,\cdot), t\in\BR\}$ be the family of bilinear
forms on $H^1_0(E)\times H^1_0(E)$ defined as
\[
a(t;\varphi,\psi)=\frac12\sum^{d}_{i,j=1}\int_Ea_{ij}(t,x)
\frac{\partial\varphi}{\partial x_i}\frac{\partial\psi}{\partial
x_j}\,dx
\]
and let  $\VV=L^2(\BR;H^1_0(E))$, $\VV'=L^2(\BR;H^{-1}(E))$,
$\WW=\{u\in\VV:\frac{\partial u}{\partial t}\in\VV'\}$. We equip
$\WW$ with the usual norm
$\|u\|_{\WW}=\|u\|_{\VV}+\|\frac{\partial u}{\partial
t}\|_{\VV'}$. By $\EE$ we  denote the time-dependent form
determined by $\AAA$, i.e.
\begin{equation}
\label{eq2.23} \EE(u,v)=\left\{
\begin{array}{l}\langle-\frac{\partial u}{\partial t},v\rangle
+\int_{\BR}a(t;u(t),v(t)),
\quad u\in\WW,v\in\VV,\smallskip \\
\langle\frac{\partial v}{\partial t}, u\rangle+
\int_{\BR}a(t;u(t),v(t)),\quad u\in\VV,v\in\WW,
\end{array}
\right.
\end{equation}
where $\langle\cdot,\cdot\rangle$ is the duality pairing between
$\VV'$ and $\VV$.

For $T>0$ we set $\HH_{0,T}=L^2(0,T;L^2(E;m))$,
$\VV_{0,T}=L^2(0,T;H^1_0(E))$, $\VV'_{0,T}=L^2(0,T;H^{-1}(E))$,
$\WW_{0,T}=\{u\in\VV_{0,T}:\frac{\partial u}{\partial
t}\in\VV'_{0,T}\}$ and $\WW_0=\{u\in\WW_{0,T}:u(0)=0\}$,
$\WW_T=\{u\in\WW_{0,T}:u(T)=0\}$. By $\EE^{0,T}$ we denote the
time-dependent form defined as
\[
\EE^{0,T}(u,v)=\left\{
\begin{array}{l}\langle-\frac{\partial u}{\partial t},v\rangle
+\int^T_0a(t;u(t),v(t)),
\quad u\in\WW_T,v\in\VV_{0,T},\smallskip \\
\langle\frac{\partial v}{\partial t}, u\rangle +
\int^T_0a(t;u(t),v(t)),\quad u\in\VV_{0,T},v\in\WW_0,
\end{array}
\right.
\]
where now $\langle\cdot,\cdot\rangle$ denote the duality pairing
between $\VV'_{0,T}$ and $\VV_{0,T}$. Note that the forms $\EE$,
$\EE^{0,T}$ can be identified with some generalized Dirichlet form
(see \cite[Example I.4.9.(iii)]{S}).

By Propositions I.3.4 and I.3.6 in \cite{S} the form $\EE^{0,T}$
determines uniquely a strongly continuous resolvents
$(G^{0,T}_{\alpha})_{\alpha>0}$,  $(\hat
G^{0,T}_{\alpha})_{\alpha>0}$ on $\HH_{0,T}$ such that
$(G^{0,T}_{\alpha})_{\alpha>0}$,  $(\hat
G^{0,T}_{\alpha})_{\alpha>0}$ are sub-Markov,
$G^{0,T}_{\alpha}(\HH_{0,T})\subset\WW_T$, $\hat
G^{0,T}_{\alpha}(\HH_{0,T})\subset\WW_0$, and
\begin{equation}
\label{eq2.14}
\EE^{0,T}_{\alpha}(G^{0,T}_{\alpha}\eta,u)=(u,\eta)_{\HH_{0,T}},\quad
\EE^{0,T}_{\alpha}(u,\hat
G^{0,T}_{\alpha}\eta)=(u,\eta)_{\HH_{0,T}}
\end{equation}
for $u\in\VV_{0,T},\eta\in\HH_{0,T}$, where
$\EE^{0,T}_{\alpha}(u,v)=\EE^{0,T}(u,v)+\alpha(u,v)_{\HH_{0,T}}$
and $(\cdot,\cdot)_{\HH_{0,T}}$ denotes the usual inner product in
$\HH_{0,T}$.

In the paper by $\mbox{cap}$ we denote the parabolic capacity
determined by the form $\EE$ (for the construction and properties
of $\mbox{cap}$ see \cite[Section 4]{O1} or \cite[Section
6.2]{O2}). We will say that some property is satisfied
quasi-everywhere (q.e. for short) if it is satisfied except for
some Borel subset of $E^1$ of capacity $\mbox{cap}$ zero. Using
$\mbox{cap}$ we define quasi-continuity as in \cite{O1}. By
\cite[Theorem 4.1]{O1} %(or \cite[Theorem 6.2.11]{O2})
each function $u\in\WW$
has a quasi-continuous $m_1$-version, which we will denote by
$\tilde u$.
%Problem: cap rownowazna def. poprzez $\WW$
%(w jedna strone p. \cite[(4.6)]{O1} (w pracy rownowaznosc nie jest
%potrzebna).

Let $\mu$ be a Borel signed measure on $E^1$. In what follows
$|\mu|$ stands for the total variation of $\mu$. By
$\MM_{0,b}(E^1)$ we  denote the set of all Borel measures on $E^1$
such that $|\mu|$ does not charge sets of zero capacity
$\mbox{cap}$ and $|\mu|(E^1)<\infty$. By $\MM_{0,b}(E_{0,T})$ we
denote the subset of $\MM_{0,b}(E^1)$ consisting of all measures
with support in $E_{0,T}$.

\subsection{Markov processes}

By general results from the theory of Markov process (see, e.g.,
\cite[Theorems 6.3.1, 6.3.10]{O2}) there exists a continuous Hunt
process $\BBM=(\Omega,(\FF_t)_{t\ge0},(\BBX_t)_{t\ge0},\zeta,
(P_{z})_{z\in E^1\cup\Delta})$ with state space $E^1$, life time
$\zeta$ and cemetery state $\Delta$ properly associated with $\EE$
in the resolvent sense. By \cite[Theorem 5.1]{O1},
\[
\BBX_{t}=(\tau(t), X_{\tau(t)}), \quad t\ge 0,
\]
where $\tau(t)$ is the uniform motion to the right, i.e.
$\tau(t)=\tau(0)+t$, $\tau(0)=s$, $P_{z}$-a.s. for $z=(s,x)$. For
an alternative construction of $\BBM$, for which the starting
point is the fundamental solution for the operator
$\frac{\partial}{\partial t}+L_t$, see \cite[Section 2]{KR:JEE}.
It is also known (this follows for instance from the construction
of $\BBM$ given in \cite{KR:JEE}) that
$\BX=\{(X,P_{s,x}):(s,x)\in\BR_+\times E)\}$, where $X$ is the
second component of $\BBX$, is a continuous time-inhomogeneous
Markov process
%The second component of $\BBX$, is a continuous time-inhommogeneous
%Markov process with the transition density $p_E$ given by
%\begin{equation}
%\label{eq2.4}
%p_E(s,x,t,y)=p(s,x,t,y)-E_{s,x}[\xi^s<t,p(s,X_{\xi^s},t-\xi^s,y)],
%\end{equation}
%where
%\[
%\xi^s=\inf\{t\ge s:X_t\notin E\}
%\] and  $p(s,x,t,y)$ is the fundamental solution for $L_t$. In
%different words, $\BX$ is a Markov process with transition density
%$p$ killed at the first exit time from $E$ and $p_E$ is the Green
%function for $\frac{\partial}{\partial t}+L_t$ on $[0,T)\times E$.
whose transition density $p_E$ is the Green function for
$\frac{\partial}{\partial t}+L_t$ on $[0,T)\times E$ (for
construction and properties of  Green's function see \cite{A}).

Let $\mu,\nu$ be  Borel measures on $E^1$ and $E$, respectively.
In what follows we write
$P_{\mu}(\cdot)=\int_{E^1}P_{s,x}(\cdot)\,d\mu(s,x)$,
$P_{s,\nu}(\cdot)=\int_{E}P_{s,x}(\cdot)\,\nu(dx)$. By $E_{\mu}$
(resp. $E_{s,\nu}$) we denote the expectation with respect to
$P_{\mu}$ (resp. $P_{s,\nu}$).

Note that if $u$ is quasi-continuous then it is
$\BBM$-quasi-continuous, i.e. for q.e. $(s,x)\in E^1$,
$P_{s,x}([0,\infty)\ni t\mapsto u(\BBX_t)$ is continuous)=1 (see,
e.g., \cite[p. 298]{O1}).

Set $\sigma_B=\inf\{t>0,\BBX_t\in B\}$. It is known (see remarks
following \cite[(5.2)]{O1}) that for a Borel set $B\subset E^1$,
$\mbox{cap}(B)=0$ if and only if $B$ is $\BBM$-exceptional, i.e.
\begin{equation}
\label{eq2.2} P_{m_1}(\sigma_{B}<\infty)=\int_{\BR}P_{s,m}(\exists
t>0:\BBX_t\in B)\,ds=0.
\end{equation}

\begin{remark}
\label{rem2.1} (i) From (\ref{eq2.2}) and the fact that
$p_E(s,x,t,\cdot)$ is strictly positive on $E$ it follows  that
$\mbox{cap}(\{s\}\times \Gamma)>0$ for any $s\in(0,T)$ and Borel
set $\Gamma\subset E$ such that $m(\Gamma)>0$. Hence, if some
property holds for q.e. $(s,x)\in (0,T)\times E$ then it holds for
$m$-a.e. $x\in E$ for every $s\in(0,T)$.
\smallskip\\
(ii) If $P_{m_1}(\sigma_{B}<\infty)=0$ then in fact
$P_{s,x}(\sigma_{B}<\infty)=0$ for q.e. $(s,x)\in E^1$. This
follows from the fact that the function $(s,x)\mapsto
P_{s,x}(\sigma_{B}<\infty)$ is excessive (see remark at the end of
\cite[p. 294]{O1}).
%(see Theorem VI.5.29 and Proposition 5.30 in MR??).
\end{remark}

Let $A$ be a positive continuous additive functional of $\BBM$ and
let $\mu\in\MM_{0,b}(E_{0,T})$ be a positive measure. We will say
that $A$ corresponds to $\mu$  (or $\mu$ corresponds to $A$) if
for q.e. $(s,x)\in[0,T)\times E$,
\begin{equation}
\label{eq2.5} E_{s,x}\int^{\zeta_{\tau}}_0f(\BBX_t)\,dA_t
=\int^T_s\!\!\int_Ef(t,y)p_E(s,x,t,y)\,\mu(dt\,dy)
\end{equation}
for every $f\in C_c([0,T]\times E)$, where
\[
\zeta_{\tau}=\zeta\wedge(T-\tau(0)).
\]
Since $p_E(s,x,\cdot,\cdot)$ is strictly positive on  $(s,T]\times
E$, if $\mu_1$ and $\mu_2$ correspond to $A$ then $\mu_1=\mu_2$.
It is also known (see, e.g., Proposition (2.12) in \cite[Chapter
IV]{BG}) that if $\mu\in\MM_{0,b}(E_{0,T})$ corresponds to $A^1$
and to $A^2$ then $A^1$ and $A^2$ are equivalent, i.e. for each
$t\in[0,T]$, $P_{s,x}(A^1_t=A^2_t)=1$ for q.e.
$(s,x)\in[0,T)\times E$. If $\mu\in\MM_{0,b}(E_{0,T})$ then we say
that $\mu$ corresponds to $A$ if $A=A^+-A^-$, where $A^+,A^-$ are
positive continuous additive functionals of $\BBM$ such that $A^+$
corresponds to $\mu^+$ and $A^-$ corresponds to $\mu^-$ (here
$\mu^+$ (resp. $\mu^-$) is the positive (resp. negative) part of
the Jordan decomposition of $\mu$). Also note that (\ref{eq2.5})
is some sort of the Revuz correspondence.

The following proposition is probably well known, but we do not
have a reference.
\begin{proposition}
\label{prop2.1} Let $\varphi\in H$, $f\in L^{2}(0,T;H)$ and let
$\tilde u$ be a quasi-continuous version of the solution
$u\in[\WW]^m$ of the Cauchy problem
\begin{equation}
\label{eq2.7} \frac{\partial u}{\partial t}+L_tu=-f,\quad
u(T)=\varphi.
\end{equation}
Then
\[
\tilde u(\BBX_{t\wedge\zeta_{\tau}})-\tilde u(\BBX_0)
=M^{[u]}_{t\wedge\zeta_{\tau}}
-\int^{t\wedge\zeta_{\tau}}_0f(\BBX_{\theta})\,d\theta,\quad
t\in[0,T],
\]
where $M^{[u]}=(M^{[u],1},\dots,M^{[u],m})$ is a continuous
martingale AF of $\BBM$ with the qua\-dra\-tic variation
\begin{equation}
\label{eq2.8} [M^{[u],k}]_t=\sum^{d}_{i,j=1}\int^{t}_0a_{ij}
\frac{\partial u^k}{\partial x_i}\frac{\partial u^k}{\partial
x_j}(\BBX_{\theta})\,d\theta,\quad t\ge0,\quad k=1,\dots,m.
\end{equation}
Moreover,
\begin{equation}
\label{eq2.9} M^{[u],k}_t=\sum^d_{i=1}\int^t_0\frac{\partial
u^k}{\partial x_i}(\BBX_{\theta})\,dM^i_{\theta},\quad t\ge0,
\end{equation}
where $M=(M^1,\dots,M^d)$ is a continuous martingale AF of $\BBM$
with the quadratic covariation
\begin{equation}
\label{eq2.11} [M^{i},M^j]_t=\sum^{d}_{i,j=1}\int^{t}_0a_{ij}
(\BBX_{\theta})\,d\theta,\quad t\ge0,\quad k=1,\dots,m.
\end{equation}
\end{proposition}
\begin{proof} We provide sketch of the proof.
Since  $\EE^{0,T}$ is a generalized Dirichlet form, it follows
from \cite[Theorem 4.5]{T} that
\[
\tilde u(\BBX_{t\wedge\zeta_{\tau}})-\tilde u(\BBX_0)
=M^{[u]}_{t\wedge\zeta_{\tau}}
+N^{[u]}_{t\wedge\zeta_{\tau}},\quad t\in[0,T],
\]
where $M^{[u]}=(M^{[u],1},\dots,M^{[u],m})$ is a continuous
martingale AF of $\BBM$ and $N^{[u]}$ is continuous AF of $\BBM$
of finite energy. Let $A_t=\int^t_0f(\BBX_{\theta})\,d\theta$,
$t\in[0,T]$. One can show that for every  $v\in\WW_T$,
\begin{equation}
\label{eq2.13} \lim_{\alpha\rightarrow\infty}\alpha^2E_{v\cdot
m_1}\int^{\zeta_{\tau}}_0 e^{-\alpha t}A_t\,dt=(f,v)_{L^2(0,T;H)}.
\end{equation}
Since $u$ satisfies (\ref{eq2.7}), we see that the left-hand side
of (\ref{eq2.13}) equals $\EE^{0,T}(u,v)$. From this and an
analogue of \cite[Theorem 7.4]{O1} for the form $\EE^{0,T}$ it
follows that $N^{[u]}_t=-A_t$, $t\in[0,T]$. Modifying slightly the
proof of \cite[(7.13)]{O1} we show (\ref{eq2.8}). Finally, to show
(\ref{eq2.9}), let us denote by $M^{u,k}$ the process on the
right-hand side of (\ref{eq2.9}) and consider a sequence $\{u_n\}$
of smooth functions such that $u_n\rightarrow u$ in $\WW$. Then by
the chain rule (see \cite[Theorem 5.5]{T}),
$M^{[u_n],k}=M^{u_n,k}$. It is clear that $\{M^{u_n,k}\}_n$ is
$e$-convergent to $M^{u,k}$. On the other hand, arguing as in the
proof of \cite[Theorem 7.2]{O1}) we show that $\{M^{[u_n],k}\}_n$
is $e$-convergent to $M^{[u],k}$, which proves (\ref{eq2.9}).
\end{proof}

Let $Y$ be a special $((\FF_t),P_z)$-semimartingale  on $[0,T]$,
i.e. an  $((\FF_t),P_z)$-semimar\-tingale admitting a (unique)
decomposition $Y_t=Y_0+M_t+B_t$, $t\in[0,T]$, with $M_0=B_0=0$,
$M$ a local $((\FF_t),P_z)$-martingale and $B$ an
$(\FF_t)$-predictable finite variation process (see, e.g.,
\cite[Section III.7]{Pr}). Recall that the $\HH^2(P_z)$ norm of
$Y$ is defined to be
$\|Y\|_{\HH^2(P_z)}=(E_z|Y_0|^2)^{1/2}+(E_z[M]_T)^{1/2}
+(E_z|B|_T^2)^{1/2}$, where  $[M]$ is the quadratic variation of
$M$ and $|B|_T$ is the variation of $B$ on the interval $[0,T]$.
By $\HH^2(P_z)$ we denote the space of  all special
$((\FF_t),P_z)$-semimartingales  on $[0,T]$ with finite
$\HH^2(P_z)$ norm.

\begin{remark}
(i) Let $z\in E_{0,T}$ and let $\varphi,f$ satisfy the assumptions
of Proposition \ref{prop2.1}. Then  $M^{[u]}$ of Proposition
\ref{prop2.1} is a martingale under $P_z$ (see \cite[p. 327]{T}),
$E_z|\tilde u(\BBX_0)|^2=|\tilde u(z)|^2<\infty$,
$E_z[M^{[u],k}_{\cdot\wedge\zeta_{\tau}}]_T<\infty$,
$E_z(\int^{\zeta_{\tau}}_0|f(\BBX_t)|\,dt)^2<\infty$ for q.e.
$z\in E_{0,T}$. Hence $Y=\tilde u(\BBX_{\cdot\wedge\zeta_{\tau}})$
is a semimartingale of class $\HH^2(P_z)$ for q.e. $z\in E_{0,T}$.
\smallskip\\
(ii) Let $M=(M^1,\dots,M^d)$ be the AF of Proposition
\ref{prop2.1}. Then $M$ is a martingale under $P_z$ for q.e. $z\in
E_{0,T}$ (see \cite[Section 5.1.2]{O2}). Set
\begin{equation}
\label{eq2.12}
B^i_t=\sum^d_{j=1}\int^t_0\sigma^{-1}_{ij}(\BBX_{\theta})\,dM^j_{\theta},
\quad t\ge0,\quad i=1,\dots,d,
\end{equation}
where $\sigma^{-1}$ is the inverse matrix of $\sigma$. By L\'evy's
theorem and (\ref{eq2.11}), for q.e. $z\in E_{0,T}$ the process
$B=(B^1,\dots,B^d)$ is under $P_z$ a $d$-dimensional standard
Brownian motion with respect to $(\FF_t)_{t\ge0}$. Finally, note
that by (\ref{eq2.9}) and (\ref{eq2.12}),
\[
M^{[u],k}_t =\sum_{i,j=1}^d\int^t_0\frac{\partial u^k}{\partial
x_i}(\BBX_{\theta})\sigma_{ij}(\BBX_{\theta})\,dB^j_{\theta},\quad
t\ge0,\quad k=1,\dots,m.
\]
\end{remark}

\section{Probabilistic solutions of the obstacle problem}
\label{sec3}

Let $\varphi:E\rightarrow\BR^m$,
$f:E_T\times\BR^m\times\BR^{m\times d}\rightarrow \BR^m$ be
measurable functions and let $D=\{D(t,x):(t,x)\in E_T\}$ be a
family of  closed convex sets in $\BR^m$ with nonempty interiors.
%and such that  $D(t,x)$ are bounded uniformly in $(t,x)\in E_T$.
Given $\varepsilon>0$ we set $D^{*}(t,x)=\{y\in
D(t,x):\dist(y,\partial D(t,x))\ge\varepsilon\}$. We will assume
that

\begin{enumerate}
\item[(A1)]$\varphi(x)\in D(T,x)$ for $x\in E$, $\varphi\in H$,

\item[(A2)]$f(\cdot,\cdot,0,0)\in L^2(0,T;H)$,

\item[(A3)]$f:E_T\times\BR^{m}\times\BR^{m\times d}\rightarrow\BR^m$
is a measurable function and there exist $\alpha,\beta\ge0$ such
that
\[
|f(t,x,y_1,z_1)-f(t,x,y_2,z_2)|
\le\alpha|y_1-y_2|+\beta\|z_1-z_2\|
\]
for all $y_1,y_2\in\BR^m$ and $z_1,z_2\in\BR^{m\times d}$.
\end{enumerate}

As for the family $D$,   we will need the following assumptions:
\begin{enumerate}
\item[(D1)] The sets $D(t,x)$ are bounded uniformly in $(t,x)\in E_T$ and the
mapping $E_T\ni(t,x)\mapsto D(t,x)\in\mbox{\rm Conv}$ is
continuous.
\item[(D2)]For
some $\varepsilon>0$ chosen so that Int$D^{*}(t,x)\neq\emptyset$
for $(t,x)\in E_T$ and some $f_{*}\in L^{2}(0,T;H)$ and
$\varphi_{*}\in H $ such that $\varphi_{*}(x)\in D^{*}(T,x)$ for
$x\in E$ there exists a solution $u_{*}\in\WW$ of the Cauchy
problem (\ref{eq1.5})
%\[
%\frac{\partial u_{*}}{\partial t}+L_tu_{*}=-f_{*},\quad
%u_{*}(T)=\varphi_{*}
%\]
such that $u_{*}(t,x)\in D^{*}(t,x)$ for q.e. $(t,x)\in E_{0,T}$.
\end{enumerate}

\begin{remark}
\label{rem3.1} Condition (D2) implies the following condition:
\begin{enumerate}
\item[(D3)]There exists a quasi-continuous function
$u_{*}:E_{0,T}\rightarrow\BR^m$ such that the process $Y^{*}$
defined as $Y^{*}_t=u_{*}(\BBX_{t\wedge\zeta_{\tau}})$,
$t\in[0,T]$, has the following properties:
\begin{enumerate}
\item[(a)]
$Y^{*}_t\in \mbox{Int}D(\BBX_{t\wedge\zeta_{\tau}})$ and
$\dist(Y^{*}_{t},\partial D(\BBX_{t\wedge\zeta_{\tau}}))
\ge\varepsilon$, $t\in[0,T]$, $P_{s,x}$-a.s. for q.e. $(s,x)\in
E_{0,T}$,
\item[(b)]
$Y^{*,i}$, $i=1,\dots,m$, is a special semimartingale of class
$\HH^2(P_{s,x})$  for q.e. $(s,x)\in[0,T)\times E$.
\end{enumerate}
\end{enumerate}
To see this, it suffices to consider a quasi-continuous version
$\tilde u_{*}$ of $u_{*}$ of condition (D2).  From
quasi-continuity of $\tilde u_{*}$, condition  (D2)  and Remark
\ref{rem3.5} it follows  that the process $Y^{*}=\tilde
u_{*}(\BBX_{\cdot\wedge\zeta_{\tau}})$ satisfies condition (a)
(see Remark \ref{rem3.5}(i) below). By Proposition \ref{prop2.1},
$Y^{*}$ admits decomposition (\ref{eq2.7}) with $u$ replaced by
$u_{*}$ and $f$ replaced by $f_{*}$. Therefore $Y^{*}$ satisfies
condition (b) by remarks following Proposition \ref{eq2.1}.
\end{remark}

Condition  (D2) is satisfied in the following natural situations.
\begin{example}
Assume that one can find  $r>0$ such that for every $(t,x)\in
E_T$, $B(0,r)\subset D(t,x)$, where $B(0,r)$ denotes the open ball
with radius $r$ and center at $0$. Then $D$ satisfies (D2). To see
this, it suffices to consider the constant function
$u_{*}=(u_{*}^1,\dots,u_{*}^m)=(0,\dots,0)$. For instance, the
above condition on $D$ is satisfied if
$D(t,x)=\{y\in\BR^m:\underline h^i(t,x)\le y_i\le \bar
h^i(t,x),\,i=1,\dots,m\}$, where $\underline h^i,\bar
h^i:[0,T]\times\bar E\rightarrow\BR$ are continuous functions such
that  $\underline h^i(t,x)<0<\bar h^i(t,x)$ for
$(t,x)\in[0,T]\times\bar E$, $i=1,\dots,m$.
\end{example}

\begin{example}
Let $D(t,x)=\{y\in\BR^m:\underline h^i(t,x)\le y_i\le \bar
h^i(t,x),\,i=1,\dots,m\}$, where $\underline h^i,\bar
h^i:[0,T]\times\bar E\rightarrow\BR$ are continuous functions such
that $\underline h^i(t,x)<\bar h^i(t,x)$ for
$(t,x)\in[0,T]\times\bar E$. Assume that there are continuous
functions $\underline h^i_0,\bar h^i_0:[0,T]\times\bar
E\rightarrow\BR$ such that $\underline h^i(t,x)\le \underline
h^i_0(t,x)<\bar h^i_0(t,x)\le \bar h^i(t,x)$ for
$(t,x)\in[0,T]\times\bar E$,  $\underline h^i_0,\bar
h^i_0\in\WW_{0,T}$ and $(\frac{\partial}{\partial
t}+L_t)\underline h^i_0,\, (\frac{\partial}{\partial t}+L_t)\bar
h^i_0\in L^2(0,T;L^2(E))$, $i=1,\dots,m$. Then (D2) is satisfied
with $u^i_{*}=(\bar h^i_0+\underline h^i_0)/2$, $i=1,\dots,m$. Of
course, the last condition for $\underline h^i_0,\bar h^i_0$ is
satisfied if $\underline h^i_0,\bar h^i_0\in
W^{1,2}_{2}((0,T)\times E)$ and $\frac{\partial a_{ij}}{\partial
x_k}\in L^{\infty}((0,T)\times E)$ for $i,j,k=1,\dots,d$.
\end{example}
%Problem. W przypadku, gdy $D(t,x)=D(x)=\{y\in\BR:\underline
%h(x)\le y\le \bar h(x)\}$ zalozenie (D2) nie sprowadza sie do
%zalozen np. w \cite[s. 405]{L2}.

In what follows we are going to show that under  assumptions
(A1)--(A3), (D1), (D2) there exists a unique solution of the
problem (\ref{eq1.1}), (\ref{eq1.2}). It is convenient to start
with probabilistic solutions. Solutions of (\ref{eq1.1}),
(\ref{eq1.2}) in the sense of the definition given in Section
\ref{sec1} will be studied in the next section. Note that the
definition formulated below is an extension, to the case of
systems, of the probabilistic definition adopted in
\cite{K:PA,KR:JEE} in case of single equation.

\begin{definition}
\label{def3.4} We say that a pair $(u,\mu)$ consisting of a
measurable  function $u=(u^1,\dots,u^m):E_{0,T}\rightarrow\BR^m$
and a measure $\mu=(\mu^1,\dots,\mu^m)$ on $E_{0,T}$ is a
probabilistic solution of the obstacle problem with data
$\varphi,f, D$ (OP$(\varphi,f,D)$ for short) if
\begin{enumerate}
\item[(a)]$u$ is quasi-continuous, $f^i_u\in L^1(E_{0,T})$,
$\mu^i\in\MM_{0,b}(E_{0,T})$, $i=1,\dots,m$,
\item[(b)] For q.e. $(s,x)\in E_{0,T}$ and $i=1,\dots,m$,
\begin{align}
\label{eq3.2} u^i(\BBX_t)&=\varphi^i(\BBX_{\zeta_{\tau}})
+\int^{\zeta_{\tau}}_{t\wedge\zeta_{\tau}}f^i_u(\BBX_{\theta})\,d\theta
+\int^{\zeta_{\tau}}_{t\wedge\zeta_{\tau}}dA^{\mu^i}_{\theta}
\nonumber\\
&\quad-\int^{\zeta_{\tau}}_{t\wedge\zeta_{\tau}}\sigma\nabla
u^i(\BBX_{\theta})\,dB_{\theta},\quad t\in[0,T], \quad
P_{s,x}\mbox{-a.s.},
\end{align}
where $A^{\mu^i}$, $i=1,\dots,m$, is the continuous additive
functional of $\BBM$ associated with $\mu^i$ in the sense of
(\ref{eq2.5}),
\item[(c)]$u(t,x)\in D(t,x)$ for q.e. $(t,x)\in E_{0,T}$ and for
every quasi-continuous function $h=(h^1,\dots,h^m)$ such that
$h(t,x)\in D(t,x)$ for q.e. $(t,x)\in E_{0,T}$ we have
\begin{equation}
\label{eq3.19} \int^{\zeta_{\tau}}_0\langle
u(\BBX_t)-h(\BBX_t),dA^{\mu}_t\rangle\le0,\quad
P_{s,x}\mbox{-a.s.}
\end{equation}
for q.e. $(s,x)\in E_T$.
\end{enumerate}
\end{definition}

\begin{remark}
\label{rem3.5} (i) From the fact that $u$ is quasi-continuous and
$u(t,x)\in D(t,x)$ for q.e. $(t,x)\in E_{0,T}$ it follows that
$u(\BBX_{t\wedge\zeta_{\tau}})\in D(\BBX_{t\wedge\zeta_{\tau}})$,
$t\in[0,T]$, $P_{s,x}$-a.s. for q.e. $(s,x)\in E_{0,T}$. To see
this, let us set $B=\{(t,x)\in E_{0,T}:u(t,x)\not\in D(t,x)\}$ and
$\sigma_B=\inf\{t>0:\BBX_t\in B\}$. Since $\mbox{cap}(B)=0$, the
set $B$ is $\BBM$-exceptional, and hence, by \ref{rem2.1}(ii),
$P_{s,x}(\sigma_B<\infty)=0$ for q.e. $(s,x)\in E_{0,T}$. Hence
$P_{s,x}(\BBX_{t\wedge\zeta_{\tau}}\in B, t\in(0,T])=0$, which
implies that $u(\BBX_{t\wedge\zeta_{\tau}})\in
D(\BBX_{t\wedge\zeta_{\tau}})$, $t\in(0,T]$, $P_{s,x}$-a.s. for
q.e. $(s,x)\in E_{0,T}$. In fact, we can replace $(0,T]$ by
$[0,T]$ , because $\BBX$ is right-continuous at $t=0$ and $D$
satisfies (D1).
\smallskip\\
(ii) Conditions (\ref{eq3.2}), (\ref{eq3.19}) of the above
definition say that under $P_{s,x}$ the pair
$(Y_t,Z_t)=(u(\BBX_{t\wedge\zeta_{\tau}}),\sigma\nabla
u(\BBX_{t\wedge\zeta_{\tau}}))$, $t\in[0,T]$, is a solution of the
generalized Markov-type reflected BSDE with final condition
$\varphi$, coefficient $f$, finite variation process $A^{\mu}$ and
obstacle $D$. %\smallskip
\\
(iii) Taking $t=0$ in (\ref{eq3.2}) and then integrating with
respect to $P_{s,x}$ we see that for  q.e. $(s,x)\in E_{0,T}$,
\[
u^i(s,x)=E_{s,x}\Big(\varphi^i(\BBX_{\zeta_{\tau}})
+\int^{\zeta_{\tau}}_0f^i_u(\BBX_{\theta})\,d\theta
+\int^{\zeta_{\tau}}_0dA^{\mu^i}_{\theta}\Big),\quad i=1,\dots,m.
\]
By this and \cite[Proposition 3.4]{K:JFA}, if $u\in L^{2}(0,T;H)$
then $u^i\in C([0,T];L^2(E))$, $i=1,\dots,m$.
\smallskip\\
(iv) From continuity of $A^{\mu^i}$ and (\ref{eq2.5}) one can
deduce that $\mu^i(\{t\}\times E)=0$ for every $t\in[0,T]$.
%\\
%K:\\
%Niech $H\subset E$ - zwarty. Niech $A_n=[v,v+\frac1n]\times H$,
%gdzie $v>s$ i niech $f=\fch_{A_n}$. Poniewaz przez aproksymacje
%mozna uogolnic (\ref{eq2.5}) do funkcji mierzalnych oraz
%$p_E(s,x,\cdot,\cdot)\ge\varepsilon>0$ na $A_1$, to mamy
%\begin{align*}
%\mu(\{v\}\times H)=\lim_{n\rightarrow\infty}\mu(A_n)&\le
%\varepsilon^{-1}\lim_{n\rightarrow\infty}\int^T_s\int_E\fch_{A_n}(t,y)
%p_E(s,x,t,y)\,\mu(dt\,dy) \\
%&=\varepsilon^{-1} \lim_{n\rightarrow\infty}E_{s,x}
%\int^{\zeta_{\tau}}_0\fch_{A_n}(\BBX_t)\,dA_t=0.
%\end{align*}
%Podobnie rozumujemy dla $v=T$. To daje wynik dla $v\in(0,T]$. Dla
%$v=0$ wynik z zalozenia, bo zakladamy, ze $\mu$ ma nosnik w
%$E_{0,T}$.\\
%KK
\end{remark}

We begin with uniqueness of probabilistic solutions.

\begin{proposition}
Assume \mbox{\rm(A3), (D1)}. Then there exists at most one
probabilistic solution of \mbox{\rm OP$(\varphi,f,D)$}.
\end{proposition}
\begin{proof}
Let $(u_1,\mu_1)$, $(u_2,\mu_2)$ be two solutions of
OP$(\varphi,f,\DD)$. Set $u=u_1-u_2$, $\mu=\mu_1-\mu_2$ and
$Y_t=u(\BBX_{t\wedge\zeta_{\tau}}), Z_t=\sigma\nabla
u(\BBX_{t\wedge\zeta_{\tau}}), K_t=A^{\mu}_{t\wedge\zeta_{\tau}}$,
$D_t=D(\BBX_{t\wedge\zeta_{\tau}})$, $t\in[0,T]$. Applying It\^o's
formula and using (A3) shows that there is $C>0$ depending only on
$\alpha,\beta$ such that for q.e. $(s,x)\in E_T$,
\[
E_{s,x}|Y_t|^2+\frac12E_{s,x}
\int^{\zeta_{\tau}}_{t\wedge\zeta_{\tau}}\|Z_{\theta}\|^2\,d\theta
\le
CE_{s,x}\int^{\zeta_{\tau}}_{t\wedge\zeta_{\tau}}|Y_{\theta}|^2\,d\theta
+2E_{s,x}\int^{\zeta_{\tau}}_{t\wedge\zeta_{\tau}}\langle
Y_{\theta} ,dK_{\theta}\rangle
\]
for all $t\in[s,T]$. By (D1) and (\ref{eq2.1}) the mapping
$E_T\times\BR^m\ni(t,x,y)\mapsto\Pi_{D(t,x)}(y)\in\BR^m$ is
continuous. Therefore the mappings
$E_T\ni(t,x)\mapsto\Pi_{D(t,x)}(u_i(t,x))\in\BR^m$, $i=1,2$, are
quasi-continuous. Using this and condition (\ref{eq3.19}) we get
\begin{align*}
\nonumber\langle Y_{t},dK_t\rangle &=\langle
Y^1_{t}-\Pi_{D_{t}}(Y^2_{t}),dK^1_t\rangle
+\langle\Pi_{D_{t}}(Y^2_{t})-Y^2_t,dK^1_t\rangle\\
&\qquad+\langle Y^2_{t}-\Pi_{D_{t}}(Y^1_{t}),dK^2_t\rangle
+\langle\Pi_{D_{t}}(Y^1_{t})-Y^1_t,dK^2_t\rangle\\
&\le|\Pi_{D_{t}}(Y^2_{t})-Y^2_{t}|\,d|K^1|_t
+|\Pi_{D_{t}}(Y^1_{t})-Y^1_{t}|\,d|K^2|_t=0.
\end{align*}
Hence
\begin{equation}
\label{eq3.1} E_{s,x}|Y_t|^2+\frac12E_{s,x}
\int^{\zeta_{\tau}}_{t\wedge\zeta_{\tau}}\|Z_{\theta}\|^2\,d\theta
\le C\int^{\zeta_{\tau}}_{t\wedge\zeta_{\tau}}
E_{s,x}|Y_{\theta}|^2\,d\theta,\quad t\in[s,T]
\end{equation}
for q.e. $(s,x)\in E_T$. Applying Gronwall's lemma yields
$E_{s,x}|Y_t|^2=0$. Since $Y^1,Y^2$ are continuous, $Y^1_t=Y^2_t$,
$t\in[0,T]$, $P_{s,x}$-a.s. for q.e. $(s,x)\in E_T$. This and
(\ref{eq3.1}) imply that $Z^1=Z^2$, $P_{s,x}\otimes dt$-a.e. on
$\Omega\times[0,T]$. That $K^1_t=K^2_t$, $t\in[0,T]$,
$P_{s,x}$-a.s. now follows from the fact that $(Y^i,Z^i,K^i)$,
$i=1,2$ satisfy the equation of condition (b) of Definition
\ref{def3.4}.
\end{proof}

\begin{proposition}
\label{prop3.7} Assume that $\varphi,f$ satisfy
\mbox{\rm(A1)--(A3)}, $D$ satisfies \mbox{\rm(D1)} and there
exists $u_{*}=(u^1_{*},\dots u^m_{*})\in[\WW]^m$ such that
$u_{*}\in D(t,x)$ for q.e. $(t,x)\in E_{0,T}$.
\begin{enumerate}
\item[\rm(i)] For every $n\in\BN$ there exists a unique strong solution
$u_n\in[\WW]^m$ of the  problem
\begin{equation}
\label{eq3.4} \frac{\partial u_n}{\partial t}+L_tu_n
=-f_{u_n}+n(u_n-\Pi_{D(\cdot,\cdot)}(u_n)),\quad u_n(T)=\varphi.
\end{equation}
\item[\rm(ii)]There is $C$ depending only on $\alpha,\beta,\Lambda$ and
$\|u_{*}\|_{[\WW]^m}$ such that
\begin{equation}
\label{eq3.5} \sup_{0\le s\le T}\|u_n(s)\|^2_H
+\int^T_0|\!|\!|\nabla u_n(t)|\!|\!|^2_H\,dt \le
C(\|\varphi\|^2_H+\|f(\cdot,\cdot,0,0)\|^2_{L^{2}(0,T;H)}).
\end{equation}
\end{enumerate}
\end{proposition}
\begin{proof}
Since $u$ is a strong solution of (\ref{eq3.4}) if and only if
$\hat u=e^{\lambda t}u$ is a strong solution of (\ref{eq3.4}) with
$L_t$ replaced by $L_t-\lambda$, $\varphi$ replaced by some
$e^{\lambda T}\varphi\in H$ and $f$ replaced by some $\hat f$
still satisfying (A2) and (A3), without loss of generality we may
replace $L_t$ in (\ref{eq3.4}) by $L_t-\lambda$. Let
$\AAA:L^2(0,T;V)\rightarrow L^2(0,T;V')$ be the operator defined
as $(\AAA v)(t)= L_tv(t)-\lambda v(t)
+f_{v(t)}-n(v(t)-\Pi_{D(\cdot,\cdot)}(v(t)))$. By (D1) and
(\ref{eq2.1}) the mapping $(t,x,y)\mapsto -n(y-\Pi_{D(t,x)}(y))$
is continuous and Lipschitz continuous in $y$ for each fixed
$(t,x)\in E_T$. Therefore for sufficiently large $\lambda>0$
(depending on $\alpha$ and $\beta$) the operator $\AAA$ is
bounded, hemicontinuous, monotone and coercive, i.e. satisfies
condition (7.84) from \cite[Chapter 2]{L2}. Therefore the
existence of a unique strong solution $u_n\in[\WW]^m$ of
(\ref{eq3.4}) follows from Theorem 7.1. and Remark 7.12 in
\cite[Chapter 2]{L2}. To prove (ii),  set $\varphi_{*}=u_{*}(T)$,
$f_{*}=-\frac{\partial u_{*}}{\partial t}-L_tu_{*}$. Then
$\varphi_{*}\in H$, $f_{*}\in L^2(0,T;V')$, $u_n-u_{*}\in[\WW]^m$,
$(u_n-u_{*})(T)=\varphi-\varphi^{*}$ and
\[
\frac{\partial(u_n-u_{*})}{\partial t}+L_t(u_n-u_{*})
=-(f_{u_n}-f_{*})+n(u_n-\Pi_{D(\cdot,\cdot)}(u_n)).
\]
Multiplying the above equation by $u_n-u_{*}$ and integrating by
parts we obtain
\begin{align*}
&\int^T_s\big\langle\frac{\partial(u_n-u_{*})}{\partial t}(t)
,u_n(t)-u_{*}(t)\big\rangle_{V',V}\,dt
-\sum^m_{k=1}\int^T_sa(t;u^k_n(t)-u^k_{*}(t),u^k_n(t)-u^k_{*}(t))\,dt\\
&\qquad=-\int^T_s\langle
f_{u_n}(t)-f_{*}(t),u_n(t)-u_{*}(t)\rangle_{V',V}\,dt\\
&\qquad\quad
+n\int^T_s(u_n(t)-\Pi_{D(\cdot,\cdot)}(u_n)(t),u_n(t)-u_{*}(t))_H\,dt.
\end{align*}
By (\ref{eq2.6}),
$(u_n(t)-\Pi_{D(\cdot,\cdot)}(u_n)(t),u_n(t)-u_{*}(t))_H\ge0$.
Therefore from the above equality and (A3) it follows that
\begin{align*}
&\frac12\|(u_n-u_{*})(s)\|^2_H+\sum^m_{k=1}\int^T_s
a(t;u^k_n(t)-u^k_{*}(t),u^k_n(t)-u^k_{*}(t))\,dt\\
&\quad\le\frac12\|\varphi-\varphi_{*}\|^2_H +\int^T_s
\|f_{*}(t)\|_{V'}\cdot\|u_n(t)-u_{*}(t)\|_V\,dt \\
&\qquad+ \int^T_s|(f(t,\cdot,0,0),u_n(t)-u_{*}(t))_H|\,dt\\
&\qquad+\int^T_s(\alpha\|u_n(t)\|_H\cdot\|u_n(t)-u_{*}(t)\|_H
+\beta|\!|\!|\nabla
u_n(t)|\!|\!|^2_H\,\cdot\|u_n(t)-u_{*}(t)\|^2_H)\,dt.
\end{align*}
Using this and standard arguments (we apply Poincar\'e's
inequality and Gronwall's lemma) shows (ii).
\end{proof}

In the proof of our main theorem on existence and approximation we
will use  some additional notation. Let $\hat\BBM=(\hat\BBX,(\hat
P_{z})_{z\in E^1\cup\Delta})$ denote a dual process associated
with the form defined by (\ref{eq2.23}) (see \cite[Theorem
5.1]{O1}). For $\mu\in\MM_{0,b}(E^1)$ let $A^{\mu}$ denote the
additive functional of $\BBM$ associated with $\BBM$ in the sense
of (\ref{eq2.5}), and let  $\hat A^{\mu}$ denote the additive
functional of $\hat\BBM$ associated with $\mu$. Given $\alpha\ge0$
and $\mu\in\MM_{0,b}(E^1)$ we set (whenever the integral exists)
\begin{equation}
\label{eq3.6} R^{0,T}_{\alpha}\mu(s,x)=
E_{s,x}\int^{\zeta_{\tau}}_0e^{-\alpha t}\,dA^{\mu}_t,\qquad \hat
R^{0,T}_{\alpha}\mu(s,x)=\hat
E_{s,x}\int^{\hat\zeta\wedge\tau(0)}_0e^{-\alpha t}\,d\hat
A^{\mu}_t
\end{equation}
for $(s,x)\in E^1$, where $\hat E_{s,x}$ denotes the expectation
with respect to $\hat P_{s,x}$ and $\hat\zeta$ is the life time of
$\hat\BBM$. By $\hat S_{00}(E_{0,T})$ we denote the set of all
$\mu\in\MM_{0,b}(E_{0,T})$ such that $|\mu|$ is a finite order
integral measure on $E^1$ (see \cite{O1} for the definition) and
$\|\hat R^{0,T}_0|\mu|\|_{\infty}<\infty$.

\begin{theorem}
\label{th3.8} Assume that $\varphi,f$ satisfy \mbox{\rm(A1)--(A3)}
and $D$ satisfies \mbox{\rm(D1), (D2)}.
\begin{enumerate}
\item[\rm(i)]There exists a unique solution $(u,\mu)$ of
\mbox{\rm OP$(\varphi,f,D)$}.
\item[\rm(ii)]
$u^i\in C([0,T];H)\cap L^2(0,T;H^1_0(E))$,
$\mu^i\in\MM_{0,b}(E_{0,T})$, $i=1,\dots,m$, and
\begin{equation}
\label{eq3.15} \sup_{0<s\le T}\|u(s)\|^2_H +\int^T_0|\!|\!|\nabla
u(t)|\!|\!|^2_H\,dt \le
C(\|\varphi\|^2_H+\|f(\cdot,\cdot,0,0)\|^2_{L^{2}(0,T;H)})
\end{equation}
with constant $C$ of Proposition \ref{prop3.7}.

\item[\rm(iii)]Let $u_n\in[\WW]^m$ be a solution of \mbox{\rm(\ref{eq3.4})}.
Then
\begin{equation}
\label{eq4.12} \|u_n-u\|^2_{L^{2}(0,T;H)}\rightarrow0, \qquad
\int^T_0|\!|\!|\nabla (u_n-u)(t)|\!|\!|^2_H\,dt \rightarrow0
\end{equation}
and
\begin{equation}
\label{eq4.8} \mu_n\rightharpoonup\mu\quad\mbox{weakly\,${}^*$ on
}(0,T]\times E.
\end{equation}
\end{enumerate}
\end{theorem}
\begin{proof}
Let $u_n\in[\WW]^m$ be a quasi-continuous version of a strong
solution of (\ref{eq3.4}). Since
$-f_{u_n}+n(u_n-\Pi_{D(\cdot,\cdot)}(u_n))\in L^2(0,T;H)$, it
follows from Proposition \ref{prop2.1} (and remark following it)
that for q.e. $(s,x)\in E_{0,T}$ the pair
\[
(Y^n_t,Z^n_t)=(u_n(\BBX_{t\wedge\zeta_{\tau}}),\sigma\nabla
u_n(\BBX_{t\wedge\zeta_{\tau}})),\quad t\in[0,T]
\]
is a solution of the  BSDE
\begin{align}
\label{eq4.1}
u_n(\BBX_{t\wedge\zeta_{\tau}})&=\varphi(\BBX_{\zeta_{\tau}})
+\int^{\zeta_{\tau}}_{t\wedge\zeta_{\tau}}
(f(\BBX_{\theta},u_n(\BBX_{\theta}),\sigma\nabla
u_n(\BBX_{\theta}))\,d\theta \nonumber\\
&\quad +\int^{\zeta_{\tau}}_{t\wedge\zeta_{\tau}}K^n_{\theta}
-\int^{\zeta_{\tau}}_{t\wedge\zeta_{\tau}}\sigma\nabla
u_n(\BBX_{\theta})\,dB_{\theta} ,\quad t\in[0,T],\quad
P_{s,x}\mbox{-a.s.},
\end{align}
where
\[
K^n_t=-n\int^{\zeta_{\tau}}_{t\wedge\zeta_{\tau}}(u_n(\BBX_{\theta})
-\Pi_{D(\BBX_{\theta})}(u_n(\BBX_{\theta})))\,d\theta.
\]
Set
\[
(Y^{*}_t,Z^{*}_t)=(u_{*}(\BBX_{t\wedge\zeta_{\tau}}),\sigma\nabla
u_{*}(\BBX_{t\wedge\zeta_{\tau}})),\quad t\in[0,T],
\]
where $u_{*}$ denotes a quasi-continuous version of the function
from condition (D2). By Remark \ref{rem3.1}, $Y^{*}$ satisfies
(D3). By It\^o's formula,
\begin{align*}
&E_{s,x}\Big(|Y^n_t-Y^{*}_t|^2
+\int^{\zeta_{\tau}}_t\|Z^n_{\theta}-\sigma\nabla
u_{*}(\BBX_{\theta})\|^2\,d\theta\Big)
\le E_{s,x}|(\varphi-\varphi_{*})(\BBX_{\zeta_{\tau}})|^2\nonumber\\
&\quad+2E_{s,x}\Big(\int^{\zeta_{\tau}}_t\langle
Y^n_{\theta}-Y^{*}_{\theta},
(f_{u_n}-f_{*})(\BBX_{\theta})\rangle\,d\theta
+\int^{\zeta_{\tau}}_t\langle Y^n_{\theta}-Y^{*}_{\theta},
dK^n_{\theta}\rangle\Big).
\end{align*}
Since by (\ref{eq2.3}) and (D3),
\[
\int^{\zeta_{\tau}}_t\langle Y^n_{\theta}-Y^{*}_{\theta},
dK^n_{\theta}\rangle \le
-\int^{\zeta_{\tau}}_t\mbox{dist}(Y^{*}_{\theta},\partial
D(\BBX_{\theta}))\,d|K^n|_{\theta}
\le-\varepsilon\int^{\zeta_{\tau}}_td|K^n|_{\theta},
\]
we have
\begin{align*}
&E_{s,x}\Big(|Y^n_t-Y^{*}_t|^2
+\int^{\zeta_{\tau}}_t\|Z^n_{\theta}-\sigma\nabla
u_{*}(\BBX_{\theta})\|^2\,d\theta
+\varepsilon\int^{\zeta_{\tau}}_td|K^n|_{\theta}\Big)\\
&\quad\le E_{s,x}|(\varphi-\varphi_{*})(\BBX_{\zeta_{\tau}})|^2 +
CE_{s,x}\Big(\int^{\zeta_{\tau}}_t
\{|Y^n_{\theta}-Y^{*}_{\theta}|^2+
|Y^n_{\theta}-Y^{*}_{\theta}|\cdot|Y^{*}_{\theta}|\\
&\qquad+|Y^n_{\theta}-Y^{*}_{\theta}|\cdot(
\|Z^n_{\theta}-Z^{*}_{\theta}\| +\|Z^{*}_{\theta}\|)
+|Y^n_{\theta}-Y^{*}_{\theta}|\cdot
|f(\BBX_{\theta},0,0)-f_{*}(\BBX_{\theta})|\}\,d\theta \Big).
\end{align*}
Using standard arguments one can deduce from this that
\begin{align}
\label{eq3.8} &E_{s,x}\Big(\sup_{0\le t\le T}
|Y^n_t-Y^{*}_t|^2+\int^{\zeta_{\tau}}_0\|Z^n_t-Z^{*}_t\|^2\,dt
+\varepsilon E_{s,x}|K^n|_{\zeta_{\tau}}\Big) \nonumber\\
&\qquad\le C
E_{s,x}\Big(|(\varphi-\varphi_{*})(\BBX_{\zeta_{\tau}})|^2\nonumber\\
&\qquad\quad+\int^{\zeta_{\tau}}_0(|f(\BBX_t,0,0)-f_{*}(\BBX_t)|^2
+|Y^{*}_t|^2 +\|Z^{*}_t\|^2)\,dt\Big).
\end{align}
Set $\mu_n=(\mu_n^1,\dots,\mu^m_n)$, where
\[
d\mu^i_n(t,y)=-n(u^i_n(t,y)-(\Pi_{D(t,y)}(u_n(t,y)))^i)\,dt\,dy.
\]
Then
\[
d|\mu^i_n|(t,y)=n|u^i_n(t,y)-(\Pi_{D(t,y)}(u_n(t,y)))^i|\,dt\,dy,
\quad |\mu_n|=\sum^m_{i=1}|\mu^i_n|.
\]
Since $u_{*}\in L^{2}(0,T;V)$, it follows from (\ref{eq3.8}) and
\cite[Proposition 3.13]{K:JFA} that
\begin{align}
\label{eq3.9} \varepsilon\sup_{n\ge1}|\mu_n|(E_{0,T})&\le
C\Big(\|\varphi-\varphi_{*}\|^2_H
+\|f(\cdot,\cdot,0,0)-f_{*}\|^2_{L^{2}(0,T;H)}\nonumber\\
&\qquad +\|u_{*}\|^2_{L^{2}(0,T;H)} +\int^T_0|\!|\!|\nabla
u_{*}(t)|\!|\!|^2_H\,dt\Big).
\end{align}
%K: Inaczej, p. notatki na koncu pliku
Set
\[
\xi=\varphi(\BBX_{\zeta_{\tau}}),\quad \bar
f(t,x,y)=f(\BBX_{t\wedge\zeta_{\tau}},y,z),\quad
D_t=D(\BBX_{t\wedge\zeta_{\tau}}),\quad t\in[0,T].
\]
We now show that  from (A1)--(A3), (D1), (D2) it follows that for
q.e. $(s,x)\in E_{0,T}$, under the measure $P_{s,x}$ the data
$\xi,\bar f$, $\DD=\{D_t,t\in[0,T]\}$ satisfy the following
hypotheses (H1)--(H4) from \cite{KRS}. In the notation of the
present paper, for fixed probability measure $P_{s,x}$,  these
hypotheses read as follows:
\begin{enumerate}
\item[(H1)]$\xi\in D_T$, $E_{s,x}|\xi|^2<\infty$,

\item[(H2)]$E_{s,x}\int_0^T|\bar f(t,0,0)|^2\,dt<\infty$,

\item[(H3)]$\bar f:[0,T]\times\Omega\times\BR^m\times\BR^{m\times d}
\rightarrow\BR^m$ is measurable with respect to
$Prog\otimes\BB(\BR^m)\otimes\BB^{m\times d}$ ($Prog$ denotes the
$\sigma$-field of all progressive subsets of $[0,T]\times\Omega$)
and there are exists $\alpha,\beta\ge0$ such that $P_{s,x}$-a.s.
we have
\[
|\bar f(t,y_1,z_1)-\bar f(t,y_2,z_2)|
\le\alpha|y_1-y_2|+\beta\|z_1-z_2\|
\]
for all $y_1,y_2\in\BR^m$ and $z_1,z_2\in\BR^{m\times d}$,

\item[(H4)]for each $N$ the mapping
$t\mapsto D_t\cap\{x\in\BR^d:|x|\le N\}\in\mbox{Conv}$ is
c\`adl\`ag $P_{s,x}$-a.s. (with the convention that $D_T=D_{T-}$),
and there is a semimartingale $A\in\HH^2(P_{s,x})$ such that
$A_t\in\mbox{Int} D_t$, $t\in[0,T]$, and $\inf_{t\le T}
\mbox{dist}(A_t,\partial D_t)>0$.
\end{enumerate}
Clearly (A1) implies that $\xi$ satisfies (H1), and (A3) implies
that $\bar f$ satisfies (H3) under $P_{s,x}$ for $(s,x)\in E_T$.
To prove (H2), let us set $g(s,x)=E_{s,x}\int^T_0|\bar
f(t,0,0)|^2\,dt$  and $B=\{(s,x)\in E_{0,T}:g(s,x)=\infty\}$. For
every $\nu\in\hat S_{00}(E_{0,T})$ we have
\begin{equation}
\label{eq3.12} \int_{E_{0,T}}g\,d\nu
=E_{\nu}\int^{\zeta_{\tau}}_0|f(\BBX_t,0,0)|^2\,dt =\int_{E_{0,T}}
R^{0,T}_0|f|^2(\cdot,\cdot,0,0)\,d\nu.
\end{equation}
Since
\[
\int_{E_{0,T}} R^{0,T}_0|f|^2(\cdot,\cdot,0,0)\,d\nu
=(|f|^2(\cdot,\cdot,0,0),\hat R^{0,T}_0\nu)_{L^2(0,T;L^2(E))}
\]
(see \cite[p. 1226]{K:JFA}), it follows from (\ref{eq3.12}) that
\[
\int_{E_{0,T}}g\,d\nu \le\|\hat
R^{0,T}_0\nu\|_{\infty}\cdot\||f|^2(\cdot,\cdot,0,0)\|_{L^1(0,T;L^1(E))}.
\]
Therefore if (A2) is satisfied then $\nu(B)=0$ for every
$\nu\in\hat S_{00}(E_{0,T})$. By this and a known analogue of
\cite[Theorem 2.2.3]{FOT}, $\widehat{\mbox{cap}}(B)=0$, where
$\widehat{\mbox{cap}}$ denotes the dual capacity (see \cite[p.
292]{O1}). Consequently, $\mbox{cap}(B)=0$, because
$\mbox{cap}(B)=\widehat{\mbox{cap}}(B)$ (see \cite[p. 292]{O1}).
This shows that if (A2) is satisfied then (H2) holds true under
$P_{s,x}$ for q.e. $(s,x)\in E_{0,T}$. That $\DD$ satisfies (H4)
under $P_{s,x}$ for q.e. $(s,x)\in E_{0,T}$ follows from (D1),
(D2) and Remark \ref{rem3.1}. Since hypotheses (H1)--(H4) are
satisfied under $P_{s,x}$ for q.e. $(s,x)\in E_{0,T}$, it follows
from  \cite[Theorem 3.6]{KRS} that for q.e. $(s,x)\in E_{0,T}$
there exists a unique solution ($Y^{s,x},Z^{s,x},K^{s,x})$ of the
reflected BSDE with final condition $\xi$, coefficient $\bar f$
and obstacle $\DD$. This means that  $Y^{s,x}_t\in D_t$ for
$t\in[s,T]$,
\begin{align}
\label{eq4.14} Y^{s,x}_t&=\varphi(\BBX_{\zeta_{\tau}})
+\int^{\zeta_{\tau}}_{t\wedge\zeta_{\tau}}
f(\BBX_{\theta},Y^{s,x}_{\theta},Z^{s,x}_{\theta})\,d\theta
\nonumber \\
&\quad +\int^{\zeta_{\tau}}_{t\wedge\zeta_{\tau}}K^{s,x}_{\theta}
-\int^{\zeta_{\tau}}_{t\wedge\zeta_{\tau}}
Z^{s,x}_{\theta}\,dB_{\theta} ,\quad t\in[0,T],\quad
P_{s,x}\mbox{-a.s.}
\end{align}
and for  every quasi-continuous  $h=(h^1,\dots,h^m)$ such that
$h(t,x)\in D(t,x)$ for q.e. $(t,x)\in E_{0,T}$ we have
\begin{equation}
\label{eq4.15} \int^{\zeta_{\tau}}_0\langle
Y^{s,x}_t-h(\BBX_t),dK^{s,x}_t\rangle\le0
\end{equation}
(observe that $h(\BBX_{\cdot\wedge\zeta_{\tau}})\in D_t$,
$t\in[0,T]$, $P_{s,x}$-a.s. by Remark \ref{rem3.5}(i)). By
\cite[Theorem 4.9]{KRS}, for any $(s,x)$ for which (\ref{eq4.1})
is satisfied for all $n\in\BN$, i.e. for  q.e. $(s,x)\in E_{0,T}$,
\begin{equation}
\label{eq3.14} \sup_{0\le t\le
T}|Y^n_t-Y^{s,x}_t|\rightarrow0,\quad
\int^T_0\|Z^n_t-Z^{s,x}_t\|^2\,dt\rightarrow0,\quad\sup_{0\le t\le
T}|K^n_t-K^{s,x}_t|\rightarrow0
\end{equation}
in probability $P_{s,x}$ as $n\rightarrow\infty$. By
(\ref{eq3.14}) and \cite[Lemma A.3.4]{FOT}  one can find versions
$Y,K$ of $Y^{s,x},K^{s,x}$ not depending on $s,x$ such that for
q.e. $(s,x)\in E_{0,T}$,
\begin{equation}
\label{eq4.5} \sup_{0\le t\le
T}|Y^n_t-Y_t|\rightarrow0,\quad\sup_{0\le t\le T}
|K^n_t-K_t|\rightarrow0
\end{equation}
in probability $P_{s,x}$.
%In fact, by (\ref{eq4.2}), the convergence hold in $L^p(P_{s,x}$
%for any $p\in[1,2)$.
From (\ref{eq3.8}) and (\ref{eq4.5}) it follows in particular that
for q.e. $(s,x)\in E_{0,T}$,
\[
E_{s,x}|u_n(s,x)-Y_0|=E_{s,x}|u_n(\BBX_0)-Y_0|\rightarrow0,
\]
i.e. $\{u_n\}$ is converges q.e. in $E_{0,T}$. Let us put
$u(s,x)=\lim_{n\rightarrow\infty}u_n(s,x)$ for those $(s,x)\in
E_{0,T}$ for which the limit exists and is finite, and $u(s,x)=0$
otherwise. Then there is an $m_1$-version of $u$ (still denoted by
$u$) such that $u$ is quasi-continuous and
\begin{equation}
\label{eq4.13} Y_t=u(\BBX_{t\wedge\zeta_{\tau}}),\quad
t\in[0,T],\quad P_{s,x}\mbox{-a.s.}
\end{equation}
for q.e. $(s,x)\in E_{0,T}$. Indeed, by (\ref{eq4.5}), $\sup_{0\le
t\le T} |u_n(\BBX_{t\wedge\zeta_{\tau}})-Y_t|\rightarrow0$ in
probability $P_{s,x}$ for q.e. $(s,x)\in E_{0,T}$. On the other
hand, since the set for which $\{u_n(s,x)\}$ does not converge is
$\BBM$-exceptional, it follows from Remark \ref{rem2.1}(ii) that
$\sup_{0\le t\le T}
|u_n(\BBX_{t\wedge\zeta_{\tau}})-u(\BBX_{t\wedge\zeta_{\tau}})|
\rightarrow0$ $P_{s,x}$-a.s. for q.e. $(s,x)\in E_{0,T}$. This
shows (\ref{eq4.13}). In particular, since $Y$ is continuous,
$[0,T]\ni t\mapsto u(\BBX_{t\wedge\zeta_{\tau}})$ is continuous
under $P_{s,x}$ for q.e. $(s,x)\in E_{0,T}$, from which it follows
that $u$ has a quasi-continuous $m_1$-version.
%(see \cite[Proposition IV.5.24]{MR}).
Since $u_n\rightarrow u$ $m_1$-a.e.  and $u_n(s,\cdot)\rightarrow
u(s,\cdot)$ $m$-a.e. for every $s\in(0,T]$ (see Remark
\ref{rem2.1}(i)), it follows from (\ref{eq3.5}) that the first
convergence in (\ref{eq4.12}) holds true,
\begin{equation}
\label{eq3.21} \|u_n(s)-u(s)\|_H\rightarrow0
\end{equation}
for every $s\in(0,T]$ and (\ref{eq3.15}) is satisfied. Moreover,
$\nabla u_n\rightarrow\nabla u$ weakly in $L^{2}(0,T;H)$ as
$n\rightarrow\infty$. By the Banach-Saks theorem there is a
subsequence $(n_k)$ such that the sequence of the Ces\`aro mean
$\bar u_N=\frac1{N}\sum^N_{k=1}u_{n_k}$ converges strongly, i.e.
\begin{equation}
\label{eq4.6} \int^T_0|\!|\!|\nabla (\bar u_N-u)(t)|\!|\!|^2_H\,dt
\rightarrow0
\end{equation}
as $N\rightarrow\infty$. For every $\nu\in\hat S_{00}(E_{0,T})$,
\begin{align}
\label{eq4.7} E_{\nu}\int^T_0\|\nabla(\bar u_N-u) (\BBX_t)\|^2\,dt
\le \|\hat R^{0,T}_0\nu\|_{\infty}\int^T_0|\!|\!|\nabla(\bar
u_N-u)(t)|\!|\!|^2_H\,dt.
\end{align}
From (\ref{eq4.6}), (\ref{eq4.7}) and a known analogue of
\cite[Theorem 2.2.3]{FOT} we conclude  that
$E_{s,x}\int^T_0\|\sigma\nabla(\bar u_N-u)
(\BBX_t)\|^2\,dt\rightarrow0$ for q.e. $(s,x)\in E_{0,T}$,
%\[
%E_{s,x}\int^T_0\|\sigma\nabla(\bar u_N-u)
%(\BBX_t)\|^2\,dt\rightarrow0
%\]
On the other hand, by (\ref{eq3.14}),
$\int^T_0\|\frac1{N}\sum^N_{k=1}Z^{n_k}_t-Z^{s,x}_t\|^2\,dt\rightarrow0$
in probability $P_{s,x}$ for q.e. $(s,x)\in E_{0,T}$.
Consequently,
\begin{equation}
\label{eq4.16} \int^T_0\|Z^{s,x}_t-\sigma\nabla
u(\BBX_t)\|^2\,dt=0,\quad P_{s,x}\mbox{-a.s.}
\end{equation}
for q.e. $(s,x)\in E_{0,T}$. From this and  (\ref{eq3.14}) we
deduce that (\ref{eq4.6}) holds with $\bar u_N$ replaced by $u_n$,
i.e. the second part of (\ref{eq4.12}) is satisfied. By
(\ref{eq3.9}) there is a subsequence (still denoted by $n$), and a
bounded measure $\mu$ on $(0,T]\times E$ such that (\ref{eq4.8})
is satisfied. By (\ref{eq2.5}), for any $f\in C_c((s,T)\times E$
and $i=1,\dots,m$ we have
\[
E_{s,x}\int^{\zeta_{\tau}}_0f(\BBX_t)\,dK^{n,i}_t
=\int^T_s\!\!\int_Ef(t,y)p_E(s,x,t,y)\,d\mu^i_n(t,y)
\]
for q.e. $(s,x)\in E_{0,T}$. Letting $n\rightarrow\infty$ and
using (\ref{eq4.5}), (\ref{eq4.8}) and the fact that  $(s,T)\times
E\ni(t,y)\rightarrow p_E(s,x,t,y)$ is continuous we obtain
\begin{equation}
\label{eq4.9} E_{s,x}\int^{\zeta_{\tau}}_0f(\BBX_t)\,dK^{i}_t
=\int^T_s\!\!\int_Ef(t,y)p_E(s,x,t,y)\,d\mu^i(t,y),\quad
i=1,\dots,m.
\end{equation}
Hence $K=A^{\mu}$ (see (\ref{eq2.5})). Thus $(Y,Z,A^{\mu})$ is a
version of $(Y^{s,x},Z^{s,x},K^{s,x})$ for q.e. $(s,x)\in
E_{0,T}$. In particular, by (\ref{eq4.14}) and (\ref{eq4.15}),
\begin{align*}
Y_t&=\varphi(\BBX_{\zeta_{\tau}})
+\int^{\zeta_{\tau}}_{t\wedge\zeta_{\tau}}
f(\BBX_{\theta},Y_{\theta},Z_{\theta})\,d\theta
\nonumber \\
&\quad +\int^{\zeta_{\tau}}_{t\wedge\zeta_{\tau}}A^{\mu}_{\theta}
-\int^{\zeta_{\tau}}_{t\wedge\zeta_{\tau}}
Z_{\theta}\,dB_{\theta},\quad t\in[0,T],\quad P_{s,x}\mbox{-a.s.}
\end{align*}
and
\[
\int^{\zeta_{\tau}}_0\langle Y_t-h(\BBX_t),dA^{\mu}_t\rangle\le0
\]
for q.e. $(s,x)\in E_{0,T}$. When combined  with (\ref{eq4.13})
and (\ref{eq4.16}) this proves that $(u,\mu)$ satisfies
(\ref{eq3.2}), (\ref{eq3.19}). What is left is to show that
$\mu\in S(E_{0,T})$. To prove this, for fixed $s\in(0,T)$ set
$\varrho(s;t,y)=\int_Ep_E(s,y,t,x)\,dx$. Observe that $(s,T]\times
E\ni(t,y)\mapsto v(s,y):=\varrho(s;t,y)$ is a weak solution of the
Cauchy-Dirichlet problem $(\frac{\partial}{\partial t}-L_t)v=0$ in
$(s,T]\times E$, $v|_{(s,T]\times\partial E}=0$, $v(s)=1$ (see
\cite[Section 8]{A}).
%\\
%K:\\
%Inaczej: Mamy
%$\varrho(s;t,y)=P_{s,y}(\zeta+s>t)=P_{s,y}(\tau_E>t)$, wiec fakt,
%ze $v$ jest rozwiazaniem sformulowanego wyzej problemy wynika z
%reprezentacji stochastycznej. \\
%KK\\
In particular, $v$ is strictly positive and continuous on
$(s,T]\times E$. Let $B$ be  Borel subset $B$ of $(s,T]\times E$
such that $\mbox{cap}(B)=0$, and let $H$ be a compact subset of
$B$. By an approximation argument, (\ref{eq4.9}) holds true for
every bounded Borel function $f$ on $(s,T]\times E$. In
particular, taking $f=\fch_H$ in (\ref{eq4.9}) we get
\begin{align*}
\int_H\varrho(s;t,y)\,d\mu^i(t,y)&=
\int_E\Big(\int_Hp_E(s,x,t,y)\,d\mu^i(t,y)\Big)\,dx\\
&= \int_E\Big(E_{s,x}\int^{\zeta_{\tau}}_0
\fch_H(\BBX_t)\,dK^{i}_t\Big)\,dx=0,
\end{align*}
the last equality being a consequence of (\ref{eq2.2}). Since
$\inf_{(t,y)\in H} \varrho(s;t,y)>0$,
%(bo $[s,T)\times E\ni(t,y)\mapsto\varrho(s;t,y)$ is strictly positive
%and continuous?)
it follows that $\mu^i(H)=0$. Since $s\in(0,T)$ was arbitrary,
$\mu^i(H)=0$ for every compact subset $H$ of $E_{0,T}$. This and
\cite[(4.4)]{O1}) imply that $\mu(B)=0$. Thus $\mu$ charges no set
of capacity zero. Hence $\mu\in\MM_{0,b}(E_{0,T})$, which
completes the proof.
\end{proof}

It is perhaps worth remarking that in case the operator $L$ is in
nondivergence form (for instance, $L$ is the Laplace operator
$\Delta$, or, more generally, $\frac{\partial a_{ij}}{\partial
x_k}\in L^{\infty}(0,T)\times E))$ for $i,j,k=1,\dots,d$), then
the process $\BBX$ corresponding to $L$ can be constructed by
solving an It\^o equation. This allows simplifying some arguments
in the proofs of the results presented above,  but actually not
much. One reason is that we are working in $L^2$ setting, so even
in the case where $L=\Delta$ the solution of (\ref{eq1.1}),
(\ref{eq1.2}) need not be continuous. On the other hand, since we
use a stochastic approach via BSDEs (the basic relation is
$Y=u(\BBX)$, where $Y$ is the first component of the solution of
the corresponding BSDE; see Remark \ref{rem3.5}(ii)), we must know
that $Y$ is continuous, and hence that $u$ is quasi-continuous.
This in turn requires the introduction of quasi-notions (capacity,
quasi-continuity, etc.) and we still have to use some results from
the probabilistic potential theory.

\section{Variational  solutions}
\label{sec4}

In this section we show that results of Section \ref{sec3} can be
translated into results on solutions of (\ref{eq1.1}),
(\ref{eq1.2}) in the sense of the analytical definition formulated
below. Solutions in the sense of this definition will be called
variational solutions.
%Note that this definition is an extension to the case of
%systems of \cite[Definition 5.1]{K:SPA}.
\begin{definition}
\label{def4.1} We say that a pair $(u,\mu)$, where
$u=(u^1,\dots,u^m):E_T\rightarrow\BR^m$ and
$\mu=(\mu^1,\dots\mu^m)$ is a signed Borel measure on $E_{0,T}$ is
a variational solution of OP$(\varphi,f,D)$ if
\begin{enumerate}
\item[(a)] $u$ is quasi-continuous,
$u^i\in C([0,T];H)\cap L^2(0,T;H^1_0(E))$,
$\mu^i\in\MM_{0,b}(E_{0,T})$, $i=1,\dots,m$,
\item[(b)]
for every $\eta=(\eta^1,\dots,\eta^m)$ such that $\eta^i$ is
bounded and $\eta^i\in\WW_0$ for $i=1,\dots,m$ we have
\begin{align}
\label{eq3.17}
&(u(t),\eta(t))_H+\int^T_t\big\langle\frac{\partial\eta}{\partial
s}(s),u(s)\big\rangle_{V',V}\,ds +\int^T_ta^{(s)}(u(s),\eta(s))\,ds\nonumber \\
&\qquad=(\varphi,\eta(T))_H +\int^T_t(f_u(s),\eta(s))_H\,ds
+\int^T_t\!\!\int_E\langle\eta,d\mu\rangle
\end{align}
for all $t\in[0,T]$, where $\{a^{(t)}(\cdot,\cdot), t\ge0\}$ is
the family of bilinear forms on $V\times V$ defined as
\[
a^{(t)}(\varphi,\psi)=\frac12\sum^m_{k=1}\sum^{d}_{i,j=1}\int_Ea_{ij}(t,x)
\frac{\partial\varphi^k}{\partial
x_i}\frac{\partial\psi^k}{\partial x_j}\,dx.
\]

\item[(c)]$u(t,x)\in D(t,x)$ for q.e. $(t,x)\in E_{0,T}$, and for
every quasi-continuous function $h=(h^1,\dots,h^m)$ such that
$h(t,x)\in D(t,x)$ for q.e. $(t,x)\in E_{0,T}$,
\begin{equation}
\label{eq3.18} \int^T_t\!\!\int_E\langle u-h,d\mu\rangle\le0
\end{equation}
for all $t\in(0,T)$.
\end{enumerate}
\end{definition}

\begin{remark}
If condition (b) of Definition \ref{def4.1} is satisfied then for
any bounded $v\in\WW_0$ we have
\begin{equation}
\label{eq4.4} \EE^{0,T}(u^i,\eta)=(\varphi^i,\eta(T))_H
+\int^T_0(f^i_u(s),\eta(s))_H\,ds
+\int_{E_{0,T}}\eta\,d\mu^i,\quad i=1,\dots,m.
\end{equation}
To see this it suffices to take $t=0$ in (\ref{eq3.17}) and
consider $\eta$ such that $\eta^i=v$, $\eta^j=0$ for $j\neq i$.
Also note that using a standard  argument (see, e.g., the
reasoning following (1.16) in \cite[Chapter III]{LSU}) and the
fact that $\mu(\{t\}\times E)=0$ for $t\in[0,T]$ (see Remark
\ref{rem3.5}(iv)) one can show that (\ref{eq4.4}) implies
(\ref{eq3.17}).
\end{remark}

\begin{remark}
Assume that $m=1$ and $D(t,x)=\{y\in\BR:\underline h(t,x)\le y\le
\bar h(t,x)\}$ for some quasi-continuous $\underline h,\bar
h:E_T\rightarrow\BR$ such that $\underline h<\bar h$. Then
(\ref{eq3.18}) reduces to the condition $\int_{E_{0,T}}
(u-\underline h)\,d\mu^+=\int_{E_{0,T}}(\bar h-u)\,d\mu^-=0$,
where $\mu^+$ (resp. $\mu^-$) is the positive (resp. negative)
part of the Jordan decomposition of $\mu$. The case of merely
measurable obstacles is discussed in \cite{K:SPA,K:PA,KR:JEE}.
\end{remark}

\begin{proposition}
Assume \mbox{\rm(A3), (D1)}. Then there exists at most one
variational solution of \mbox{\rm OP$(\varphi,f,D)$}.
\end{proposition}
\begin{proof}
Suppose that $(u_1,\mu_1),(u_2,\mu_2)$ are solutions of
OP$(\varphi,f,D)$ and set $u=u_1-u_2$, $\mu=\mu_1-\mu_2$. Then
from (\ref{eq3.17}) it follows that for every bounded
quasi-continuous $\eta\in\WW_0$,
\begin{align}
\label{eq4.17} &(u(t),\eta(t))_H
+\int^T_t\big\langle\frac{\partial\eta}{\partial s},
u(s)\big\rangle_{V',V}\,ds +\int^T_ta^{(s)}(u(s),\eta(s))\,ds\nonumber\\
&\qquad=\int^T_t(f_{u_1}(s)-f_{u_2}(s),\eta(s))_H\,ds
+\int^T_t\!\!\int_E\langle\eta,d(\mu_1-\mu_2)\rangle.
\end{align}
Set $u^i_{\alpha}=\alpha \hat R^{0,T}_{\alpha}u^i$, $i=1,\dots,m$,
where $\hat R^{0,T}_{\alpha}u^i$ is defined by (\ref{eq3.6}) with
$\mu=u^i\,dm_1$. By \cite[(3.17)]{K:JFA}, $\hat
R^{0,T}_{\alpha}u^i$ is an $m_1$-version of $\hat
G^{0,T}_{\alpha}u^i$, where $(\hat G^{0,T}_{\alpha})_{\alpha>0}$
is the co-resolvent associated with the form $\EE^{0,T}$ (see
(\ref{eq2.14})).  Taking
$\eta=u_{\alpha}=(u^1_{\alpha},\dots,u^m_{\alpha})$ as test
function in (\ref{eq4.17}) we obtain
\begin{align}
\label{eq4.3} &(u(t),u_{\alpha}(t))_H
+\int^T_t\big\langle\frac{\partial u_{\alpha}}{\partial
s}(s),u(s)\big\rangle_{V',V}\,ds
+\int^T_ta^{(s)}(u(s),u_{\alpha}(s))\,ds
\nonumber\\
&\qquad=\int^T_t(f_{u_1}(s)-f_{u_2}(s),u_{\alpha}(s))_H\,ds
+\int^T_t\!\!\int_E\langle u_{\alpha},d(\mu_1-\mu_2)\rangle.
\end{align}
The generator $\hat\LL$ of the co-resolvent $(\hat
G^{0,T}_{\alpha})_{\alpha>0}$ has the form   $(\hat\LL
v)(s)=-\frac{\partial v}{\partial s}(s)+L_tv(s)$ for $v\in\WW_0$.
Hence $(\alpha-(-\frac{\partial}{\partial s}+L_s))\hat
R^{0,T}_{\alpha}u=u $, and consequently
\[
(-\frac{\partial}{\partial s}+L_s)u_{\alpha}=\alpha(u_{\alpha}-u).
\]
Therefore
\[
\int^T_t\big\langle\frac{\partial u_{\alpha}}{\partial s}(s),
u(s)\big\rangle_{V',V}\,ds
+\int^T_ta^{(s)}(u(s),u_{\alpha}(s))\,ds
=\int^T_t\alpha(u(s)-u_{\alpha}(s),u(s))\,ds
\]
and
\begin{align*}
 &\frac12\|u_{\alpha}(T)\|^2_H-\frac12\|u_{\alpha}(t)\|^2_H
+\int^T_ta^{(s)}(u_{\alpha}(s),u_{\alpha}(s)(s))\,ds\\
&\qquad=\int^T_t\alpha(u(s)-u_{\alpha}(s),u_{\alpha}(s))\,ds.
\end{align*}
By the last two equalities,
\begin{align}
\label{eq4.20} &\int^T_t\big\langle\frac{\partial
u_{\alpha}}{\partial s}(s), u(s)\big\rangle_{V',V}\,ds
+\int^T_ta^{(s)}(u(s),u_{\alpha}(s))\,ds \nonumber\\
&\qquad=\int^T_t\alpha(u(s)-u_{\alpha}(s),
u(s)-u_{\alpha}(s))\,ds+\int^T_t\alpha(u(s)-u_{\alpha}(s),
u_{\alpha}(s))\,ds\nonumber\\
&\qquad\ge\int^T_ta^{(s)}(u_{\alpha}(s),u_{\alpha}(s))\,ds
+\frac12\|u_{\alpha}(T)\|^2_H-\frac12\|u_{\alpha}(t)\|^2_H.
\end{align}
Since $u\in L^2(0,T;V)$ and it is known that $(\hat
G^{0,T}_{\alpha})$ is strongly continuous (see, e.g.,
\cite[Proposition I.3.7]{S}), $u_{\alpha}\rightarrow u$ in
$L^2(0,T;V)$ as $\alpha\rightarrow\infty$. Hence, for every
$t\in[0,T]$,
\begin{equation}
\label{eq4.21} \int^T_ta^{(s)}(u_{\alpha}(s),u_{\alpha}(s))\,ds
\rightarrow \int^T_ta^{(s)}(u(s),u(s))\,ds.
\end{equation}
Since $u$ is quasi-continuous, $[0,T]\ni t\mapsto u(\hat\BBX_t)$
is continuous $P_{s,x}$-a.s. for q.e. $(s,x)\in E_{0,T}$. Hence
$\alpha\int^{\hat\zeta\wedge\tau(0)}_0e^{-\alpha t}
u(\hat\BBX_t)\,dt=\int_0^{\alpha(\hat\zeta\wedge\tau(0))}
e^{-t}u(\hat\BBX_{t/\alpha})\,dt \rightarrow u(\BBX_0)$
$P_{s,x}$-a.s. as $\alpha\rightarrow\infty$ for q.e. $(t,x)\in
E_{0,T}$. Since $u(t,y)\in D(t,y)$ for q.e. $(t,y)\in E_{0,T}$ and
the sets $D(t,y)$ are uniformly bounded, it follows from this and
the Lebesgue dominated convergence theorem that
$u_{\alpha}(s,x)\rightarrow E_{s,x}u(\BBX_0)=u(s,x)$ for q.e.
$(s,x)\in E_{0,T}$. Hence, by the Lebesgue dominated convergence
theorem again,
\begin{equation}
\label{eq4.22} (u(t),u_{\alpha}(t))_H\rightarrow
\|u(t)\|^2_H,\qquad \|u_{\alpha}(t)\|^2_H\rightarrow \|u(t)\|^2_H
\end{equation}
and
\begin{equation}
\label{eq4.23} \int^T_t\!\!\int_E\langle
u_{\alpha},d(\mu_1-\mu_2)\rangle
\rightarrow\int^T_t\!\!\int_E\langle u,d(\mu_1-\mu_2)\rangle.
\end{equation}
By (\ref{eq3.18}), the right-hand side of (\ref{eq4.23}) is
nonpositive. Therefore letting $\alpha\rightarrow\infty$ in
(\ref{eq4.3}) and using (\ref{eq4.20})--(\ref{eq4.23}) we get
\[
\frac12\|u(t)\|^2_H+\int^T_ta^{(s)}(u(s),u(s))\,ds\le
\int^T_t(f_{u_1}(s)-f_{u_2}(s),u(s))_H\,ds,\quad t\in[0,T].
\]
Using  (A3) and Gronwall's lemma we deduce from the above
inequality that  %$\|u(t)\|^2_H=0$, $t\in[0,T]$, hence
$u=0$ a.e.
\end{proof}

\begin{theorem}
Under the assumptions of Theorem \ref{th3.8} there exists a unique
variational solution of \mbox{\rm OP$(\varphi,f,D)$}.
\end{theorem}
\begin{proof}
We need only prove the existence of a solution. By Theorem
\ref{th3.8} there exists a probabilistic solution $(u,\mu)$ of
OP$(\varphi,f,D)$. By (\ref{eq3.15}) and Remark
\ref{rem3.5})(iii), $u\in C([0,T];H)$, and hence $(u,\mu)$
satisfies condition (a) of Definition \ref{def4.1}.  To see that
(b) and (c) are satisfied, let us consider solutions $u_n$ of
(\ref{eq3.4}). Then for every $\eta=(\eta^1,\dots,\eta^m)$ such
that $\eta_i\in C_c(E_{0,T})$ for $i=1,\dots,m$ we have
\begin{align*}
&(u_n(t),\eta(t))_H
+\int^T_t\big\langle\frac{\partial\eta}{\partial s}(s),
u_n(s)\big\rangle_{V',V}\,ds
+\int^T_ta^{(s)}(u_n(s),\eta(s))\,ds\nonumber \\
&\qquad=(\varphi,\eta(T))_H +\int^T_t(f_{u_n}(s),\eta(s))_H\,ds
+\int^T_t\!\!\int_E\eta\,d\mu_n,
\end{align*}
Letting $n\rightarrow\infty$ in the above equality and using
(\ref{eq4.12}), (\ref{eq4.8}), (\ref{eq3.21}) and (A3) we conclude
that (\ref{eq3.17}) is satisfied for $\eta$ as above, and hence
for $\eta\in\WW_0$ since  $C_c(E_{0,T})$ is dense is $\WW_0$.
Moreover, by (\ref{eq2.5}) and (\ref{eq3.19}), for every
quasi-continuous $h$ such that $h(t,x)\in D(t,x)$ for q.e.
$(t,x)\in E_{0,T}$ we have
\[
\sum^m_{i=1}\int^T_s\!\!\int_E(u^i-h^i)(t,y)p_E(s,x,t,y)\,d\mu^i(t,y)
=E_{s,x}\int^{\zeta_{\tau}}_0\langle
u(\BBX_t)-h(\BBX_t),dA^{\mu}_t\rangle\le0
\]
for q.e. $(s,x)\in E_{0,T}$. Integrating with respect to $x$ and
using Remark \ref{rem2.1}(i) we obtain
\[
\sum^m_{i=1}\int^T_s\!\!\int_E(u^i-h^i)(t,y)
\varrho(s;t,y)\,d\mu^i(t,y)\le0
\]
for every $s\in(0,T]$, which implies (\ref{eq3.18}), because
$\varrho(s;\cdot,\cdot)$ is strictly positive.
\end{proof}

\begin{remark}
Let the assumptions of Theorem \ref{th3.8} hold. Then the first
component $u$ of the solution of OP$(\varphi,f,D)$ satisfies the
following parabolic variational inequality: for every $v\in\WW$
such that $v(t,x)\in D(t,x)$ for a.e. $(t,x)\in E_T$,
\begin{align}
\label{eq3.20} &\int^T_0\big\langle\frac{\partial v}{\partial
t}(t),v(t)-u(t)\big\rangle_{V',V}\,dt
-\int^T_0a^{(t)}(u(t),v(t)-u(t))\,dt\nonumber\\
&\qquad +\int^T_0(f_u(t),v(t)-u(t))_H\,dt
\le\frac12\|v(T)-\varphi\|^2_H.
\end{align}
To see this,  let us consider the solution $u_n$  of
(\ref{eq3.4}). Taking $v-u_n$ as test function in (\ref{eq3.4}) we
get
\begin{align}
\label{eq4.10} &\int^T_0\big\langle\frac{\partial u_n}{\partial
t},v(t)-u_n(t)\big\rangle_{V',V}\,dt
-\int^T_ta^{(t)}(u_n(t),v(t)-u_n(t))\,dt\nonumber\\
&\qquad=-\int^T_0(f_{u_n}(t),v(t)-u_n(t))_H\,dt
-\int_{E_{0,T}}\langle v-u_n,\,d\mu_n\rangle.
\end{align}
Since
\begin{align*}
\int^T_0\big\langle\frac{\partial u_n}{\partial
t}(t),v(t)-u_n(t)\big\rangle_{V',V}\,dt
&=\int^T_0\big\langle\frac{\partial v}{\partial
t}(t),v(t)-u_n(t)\big\rangle_{V',V}\,dt\\
&\quad-\frac12\|v(T)-\varphi\|^2_H+\frac12\|v(0)-u_n(0)\|^2_H,
\end{align*}
it follows from (\ref{eq4.10}) that
\begin{align*}
&\int^T_0\big\langle\frac{\partial v}{\partial
t}(t),v(t)-u_n(t)\big\rangle_{V',V}\,dt
-\int^T_0a^{(t)}(u_n(t),v(t)-u_n(t))\,dt\\
&\qquad\quad+\int^T_0(f_{u_n}(t),v(t)-u_n(t))_H\,dt\\
&\qquad=\frac12\|v(T)-\varphi\|^2_H-\frac12\|v(0)-u(0)\|^2_H
-\int_{E_{0,T}}\langle v-u_n,\,d\mu_n\rangle\le
\frac12\|v(T)-\varphi\|^2_H,
\end{align*}
the last inequality being a consequence of (\ref{eq2.6}). Letting
$n\rightarrow\infty$ and using (\ref{eq4.12}), (\ref{eq4.8}) and
(A3) we get (\ref{eq3.20}).
\end{remark}

We close this section with stating some open problems. In the
whole paper we assume that $E$ is bounded. Therefore our results
do not apply to single equations with one obstacle or, for
instance,  to systems of equations with obstacles of the form
\begin{equation}
\label{eq4.18} D(t,x)=\{y\in\BR^m:y_i\ge h^i(t,x)\}\quad
\mbox{or}\quad D(t,x)=\{y\in\BR^m:y_i\le h^i(t,x)\}.
\end{equation}
It would be desirable to extend our results to the setting which
cover these examples. By the way, note that in the paper
\cite{KRS}, on which we rely, $E$ need not be bounded (see
condition (H4) formulated in the proof of Theorem \ref{th3.8}).
Another  natural problem is to generalize the results of the
present paper to systems involving more general operators, for
instance nonsymmetric operators of the form
$L^b_t=L_t+\sum^d_{i=1}b_i(x)\frac{\partial}{\partial x_i}$ with
bounded measurable $b:[0,T]\times E\rightarrow\BR^d$. It is worth
noting here that in \cite[Section 1.2]{MP} the existence of
solutions of variational inequalities with obstacles of the form
(\ref{eq4.18}) and operator $L^b_t$ is proved in the special case
where $f=(f^1,\dots,f^m)$ and
$f^i(t,x,y)=\sum^m_{j=1}c_{ij}(t,x)y_j$ with bounded measurable
$c_{ij}$ such that $c_{ij}\ge0$ a.e. for $i\neq j$ (in fact, in
\cite{MP}, the principal part and the first order part of the
operator need not be the same in each  equation of the system).
%the principal parts and the first order parts of
%the operators appearing in equations of the system may be
%different).
Still another problem of interest is to generalize the results of
the paper to irregular obstacles and/or $L^1$ data (for one
dimensional results in this direction see
\cite{K:SPA,K:PA,KR:JEE}).

\smallskip
\noindent{\bf Acknowledgements}
\smallskip\\
This research was supported by  NCN Grant No. 2012/07/B/ST1/03508.

\end{document}